\documentclass{article}

\usepackage[backref=page]{hyperref}
\usepackage[bibtex-style]{amsrefs}
\usepackage{amsmath, amsthm, amssymb, amsfonts}
\usepackage[utf8]{inputenc}
\usepackage[T1]{fontenc}
\usepackage{bm}
\usepackage{tikz}
\usepackage[capitalize]{cleveref}
\usetikzlibrary{calc}
\usetikzlibrary{decorations}
\usetikzlibrary{decorations.pathreplacing}
\usetikzlibrary{shapes}
\usetikzlibrary{arrows, automata, decorations.pathreplacing, fit, matrix, patterns, positioning}
\usepackage{tikz-qtree}
\usepackage{xifthen} 
\usepackage{xspace}
\usepackage{bbold}
\usepackage[mathscr]{euscript}
\usepackage{enumitem}
\usepackage{todonotes}
\usepackage{pgf}

\usepackage{lineno}

\hypersetup{colorlinks=true, linkcolor=blue, urlcolor=blue, citecolor=blue, 
pdfpagemode=UseNone, pdfstartview=} 

\newcommand\absdot[2]{
	\node at #1 {\normalsize $\bullet$};
	\node at #1 [below] {$#2$};
}

\newcommand{\plotperm}[1]{
	\foreach \j [count=\i] in {#1} {
		\absdot{(\i,\j)}{};
	};
}

\newcommand{\plotpermgraph}[1]{
	\foreach \j [count=\i] in {#1} {
		\foreach \b [count=\a] in {#1} {
			\ifthenelse{\a<\i \AND \b>\j}{\draw (\a,\b)--(\i,\j);}{}
		};
	};
	\plotperm{#1};
}

\newcommand{\suf}{\leftarrow} 

\newcommand{\A}{{\mathcal A}}

\newcommand{\C}{{\mathcal C}}

\renewcommand{\L}{{\mathcal L}}
\newcommand{\coL}{\L^{\suf}} 

\renewcommand{\P}{{\mathcal P}}
\renewcommand{\S}{{\mathcal S}}

\newcommand{\U}{{\mathcal U}}

\renewcommand{\a}{\mathsf{a}} 
\renewcommand{\b}{\mathsf{b}} 

\newcommand{\bbP}{\mathbb{P}} 
\newcommand{\bbE}{\mathbb{E}} 
\newcommand{\bbI}{\mathbb{1}} 

\newcommand{\cO}{\mathcal{O}} 

\DeclareMathOperator{\weight}{wt}

\newtheorem{theorem}{Theorem}[section]
\newtheorem{corollary}[theorem]{Corollary}
\newtheorem{lemma}[theorem]{Lemma}

\newtheorem{proposition}[theorem]{Proposition}
\newtheorem{fact}[theorem]{Fact}

\theoremstyle{definition}
\newtheorem{definition}[theorem]{Definition}
\newtheorem{observation}[theorem]{Observation}
\newtheorem{question}[theorem]{Question}

\newcommand{\Av}{\operatorname{Av}}

\newcommand{\first}{\operatorname{first}}
\newcommand{\last}{\operatorname{last}}

\newcommand{\we}{\equiv}
\newcommand{\wwe}{\equiv}

\renewcommand{\leq}{\leqslant}

\title{Wilf collapse in permutation classes}
\author{Michael Albert\thanks{University of Otago, Dunedin, New Zealand.} \and V\'{i}t 
Jel\'{i}nek\thanks{Computer Science Institute, Charles University, Prague, Czechia. Supported 
by project Impuls of the Neuron Fund for
Support of Science, by project SVV–2017–260452, and by project 18-19158S of the Czech Science 
Foundation.} \and Michal 
Opler\footnotemark[2]}

\begin{document}
\maketitle

\begin{abstract}
For a hereditary permutation class $\C$, we say that two permutations $\pi$ and $\sigma$ of $\C$ 
are \emph{Wilf-equivalent in $\C$}, if $\C$ has the same number of permutations avoiding $\pi$ as 
those avoiding~$\sigma$. We say that a permutation class $\C$ exhibits a \emph{Wilf collapse} if 
the number of permutations of size $n$ in~$\C$ is asymptotically larger than the number of 
Wilf-equivalence classes formed by these permutations.

Previously, only a few specific examples of classes were known to exhibit Wilf collapse. In this 
paper, we show that Wilf collapse is a surprisingly common phenomenon. Among other results, we show 
that Wilf collapse occurs in any permutation class with unbounded growth and finitely many 
sum-indecomposable permutations. 

Our proofs are based on encoding the elements of a permutation class $\C$ as words 
and analyzing the structure of a random permutation in $\C$ using this 
representation.
\end{abstract}

\section{Introduction}

Given a collection, $\C$, of finite structures one associates with it the \emph{growth function} 
$n\mapsto c_n$ where $c_n$ is the number of structures in $\C$ of size $n$. There seems to be no 
generally accepted word for the concept of ``two classes having the same growth function'' -- we 
have decided to say that such classes are \emph{rank-equinumerous}. In the study of permutation 
classes (exact definitions follow in the next section) much attention has been paid to examples of 
rank-equinumerosity, perhaps the most famous being that the collection of permutations that do not 
contain the permutation $231$ as a subpermutation, and the class of those that do not contain $321$ 
are rank-equinumerous. While such equivalences are interesting they are perhaps not too surprising 
given that simple questions tend to have simple answers and there are only so many simple answers to 
go around.

We are concerned with a special sort of rank-equinumerosity. This arises when we begin with a 
universe, $\U$, of finite structures carrying a containment relation denoted~$\leq$. The collections 
we then consider are down-sets in $\U$, i.e., subcollections of $\U$ closed under containment 
(sometimes called \emph{hereditary} subsets of~$\U$). Even more specifically, we consider only those 
down-sets that are defined by the \emph{avoidance} of a single structure $A$, i.e., they consist of 
all the elements of $\U$ that do not contain~$A$. We then say that $A$ and $B$ are 
\emph{Wilf-equivalent} (in $\U$ if the context is not clear), if the down-set of structures 
avoiding $A$ is rank-equinumerous to the down-set of structures avoiding~$B$. For example, the 
rank-equinumerosity mentioned in the previous paragraph arises then when $\U$ is the set of all 
permutations, $A = 231$ and $B = 321$.

We will also say that $\U$ exhibits a \emph{Wilf collapse} if the number of Wilf-equivalence classes 
on structures of size $n$ is small when compared to the total number of structures of size $n$, 
i.e., the average size of a Wilf class tends to infinity as $n$ grows. We further say that $\U$ 
exhibits an \emph{exponential Wilf collapse} if the average size of a Wilf class is exponential 
in~$n$.

While there have been many previous investigations that deal with specific examples of Wilf 
equivalence, or even a few general groups of Wilf-equivalent structures, there has been relatively 
little attention paid to the phenomenon of Wilf collapse. In \cite{AlbertBouvel} it was demonstrated 
that the universe of 312-avoiding permutations exhibits a Wilf collapse, and in \cite{AlbertLi} that 
every permutation class with two basis elements of size 3 and itself having an unbounded growth 
function (which is of course a prerequisite for Wilf collapse!) exhibits a Wilf collapse. Notably, 
as of this moment, we still do not know if the universe of 321-avoiding permutations exhibits a Wilf 
collapse and the results of this paper do not speak to this case.

In this paper, rather than focusing on individual examples of permutation classes, we derive general 
structural criteria that imply Wilf collapse, or even exponential Wilf collapse. Our approach is 
based on decomposing permutations into indecomposable blocks using the sum operation 
(see Section~\ref{sec-basic} for precise definitions). Specifically, we can prove the following 
results.
\begin{itemize}
 \item Any permutation class $\C$ obtained as a sum-closure of finitely many
permutations exhibits an exponential Wilf collapse, except for the class of 
21-avoiding permutations, whose growth function is bounded. See 
Corollary~\ref{cor-sumclosedfinitealphabet}.
 \item Any permutation class $\C$ with unbounded growth function and with only finitely many 
indecomposable permutations has a Wilf collapse. See Theorem~\ref{thm-finind}.
\end{itemize}
We remark that the first of these results is in fact a special case of a more general theorem 
(Theorem~\ref{thm-sumclosed}), which deals with sum-closures of possibly infinite sets satisfying 
certain additional restrictions. 

While our results focus on permutation classes, the underlying arguments can be generalised easily 
to some other contexts. The basis of our approach is the observation that any permutation can be 
uniquely expressed as a sum of a sequence of sum-indecomposable components. This yields a 
representation of a permutation by a word over an alphabet consisting of the indecomposable 
permutations. The containment of permutations then corresponds to a certain ``greedy'' embedding of 
words. We then identify, for a permutation class $\C$ satisfying suitable closure properties, a 
number of ``local modifications'' of the word which preserve the Wilf class of the corresponding 
permutation. These local modifications often take the form of applying a symmetry operation to a 
subword. Finally, and this is usually the most difficult part of the argument, we analyse the 
structure of a word representing a random permutation $\pi$ of $\C$, and show that with high 
probability it offers many opportunities for such local modifications, showing that $\pi$ belongs to 
a large Wilf class. 
%
%

Our emphasis in this paper is simply on establishing the existence of a Wilf collapse: we make no 
attempt to determine the precise number of Wilf classes, or even an accurate asymptotic estimate. 
This is because our results are necessarily based on general criteria for Wilf equivalence, while 
specific permutation classes may often admit additional rules or coincidences that cause further 
collapse. Besides, even in quite simple settings such as those considered in \cite{AlbertLi} where 
the precise nature of a collapse can be computed, dealing with the exact answers can become quite 
technical. That is, demonstrating that certain groups of structures are Wilf-equivalent is easy, but 
demonstrating that no others are seems hard. Similarly, in \cite{AlbertBouvel} there is a 
conjectural description of the exact nature of the Wilf collapse within certain classes enumerated 
by the Catalan numbers, and while the experimental evidence in its favour seems quite strong, there 
is no known way to rule out some other ``accidental'' coincidences.

The structure of the remainder of this paper is as follows. In Section~\ref{sec-basic}, we provide 
the basic definitions needed to discuss permutation classes and Wilf collapse. 
Section~\ref{sec-words} then carries out some necessary preparatory work about words and 
generalisations of the subword relation. Section~\ref{sec-wilf} is devoted to the statements and 
proofs of our main results. Our two main results there are Theorem~\ref{thm-sumclosed} which deals 
with sum-closed classes, and Theorem~\ref{thm-finind} concerning classes having only finitely many 
sum-indecomposable permutations. Finally, in Section~\ref{sec-concluding} we discuss the 
significance and limitations of our results, and pose some further questions which we consider 
pertinent.

\section{Basic definitions}
\label{sec-basic}

We refer the reader to Vatter's excellent survey \cite{Vatter-survey} for a much more detailed consideration of permutation classes (as well as an historical introduction) providing here only the essential elements for our work. We are concerned only with permutations of size $n$ which we generally think of in one-line notation i.e., as sequences of length $n$ consisting of the elements of $[n] = \{1,2,\dots,n\}$ in some order. We write $|\pi|$ for the size of a permutation $\pi$\footnote{We generally try to avoid using the word ``length'' here although it is quite common and natural due to a possible confusion with the notion of length of a permutation arising in algebraic combinatorics.}.

When we take a subsequence of size $k$ of such a sequence and then relabel it so that its least element is labelled 1, its second least element 2, \dots, and its greatest element $k$ then we obtain another permutation and this relationship defines the notion of containment between permutations (sometimes called ``containment as patterns''). To rephrase: a permutation $\tau$ of size $n$ \emph{contains} a permutation $\pi$ of size $k$ if there is a subsequence of $\tau$ consisting of $k$ elements whose relabelling by relative value yields $\pi$. If this occurs we write $\pi \leq \tau$, and if not we say that $\tau$ \emph{avoids} $\pi$ and write $\pi \not\leq \tau$. For instance the permutation $31524$ contains the patterns $123$ (as $124$)  and $213$ (as either $315$ or $314$) but not the pattern $321$ (since no three of its elements form a descending sequence). A \emph{permutation class} is a collection of permutations, $\C$, closed downwards under containment, i.e., if $\tau \in \C$ and $\pi \leq \tau$ then $\pi \in \C$.

The partially ordered set $\S$ of all finite permutations ordered by containment admits eight symmetries corresponding to the action of the dihedral group on a square. These symmetries are easy to understand if we think of a permutation $\pi$ as being represented by the set of points $(i, \pi_i)$ contained in an axis-aligned square. Reflection in a vertical axis is called ``reverse'', in a horizontal axis ``complement'', and in an upward sloping diagonal ``inverse''. 

Given a permutation class $\C$ other than the class of all permutations, there are some 
$\leq$-minimal permutations in its complement and these are called its \emph{basis}. So $\C$ can 
also be described as the set of all permutations avoiding any permutation in its basis. If $X$ is 
any set of permutations then we write $\Av(X)$ for the class of permutations that avoid every 
element of $X$. If $X$ is an antichain with respect to containment then $X$ will be the basis of 
$\Av(X)$.

Given two permutations $\alpha$ and $\beta$ define their \emph{sum} $\alpha \oplus \beta$ to be the 
concatenation of $\alpha$ and $a + \beta$ where $a$ is the size of $\alpha$. For instance $231 
\oplus 2413 = 2315746$. It is easy to see that this operation is associative on permutations. A 
class $\C$ is \emph{sum-closed} if whenever $\alpha, \beta \in \C$ then also $\alpha \oplus \beta 
\in \C$. A permutation is \emph{sum-indecomposable} if it cannot be written as a proper sum of two 
permutations. There is a dual notion of \emph{skew-sum} ($\ominus$) where $\alpha \ominus \beta$ is 
the concatenation of $b + \alpha$ with $\beta$ (where the size of $\beta$ is $b$). It is easy to see 
that a permutation class is sum- (resp.~skew-) closed if and only if all of its basis elements are 
sum- (resp.~skew-) indecomposable.


It is particularly convenient to work with (and within) sum-closed classes because there is a natural representation of any permutation $\alpha$ in such a class as the unique sequence of sum-indecomposable permutations $\alpha_1 \alpha_2 \cdots \alpha_k$ for which 
\[
\alpha = \alpha_1 \oplus \alpha_2 \oplus \cdots \oplus\alpha_k.
\]
This identifies the class with the language of words over its sum-indecomposables. Given any set $X$ 
of permutations there is a smallest class $\C$ which is sum-closed and contains $X$ (obtained simply 
by finding all the sum-indecomposable permutations that are contained in some element of $X$ and 
then taking all sums of those). This class is called the \emph{sum-closure} of~$X$. 

Let a class $\C$ be given. Two permutations $\alpha, \beta \in \C$ are \emph{Wilf-equivalent in 
$\C$} (written $\alpha \we_{\C} \beta$) if the two classes $\C \cap \Av(\alpha)$ and $\C \cap 
\Av(\beta)$ are rank-equinumerous, i.e., have the same growth functions. The equivalence classes of 
$\we_\C$ are known as the \emph{Wilf classes}.

\begin{observation}
If $\alpha \we_{\C} \beta$ then $|\alpha| = |\beta|$.
\end{observation}

\begin{proof}
Without loss of generality suppose that $k = |\alpha| \le |\beta|$. The number of permutations in 
$\C \cap \Av(\alpha)$ of size $k$ is exactly one less than the number of permutations in $\C$ of 
size $k$ and for this to be true of $\C \cap \Av(\beta)$ as well we must have $|\beta| = k$, 
otherwise every permutation in $\C$ of size $k$ belongs to $\C \cap \Av(\beta)$.
\end{proof}

For a positive integer $n$, let $\C_n$ denote the set of permutations in $\C$ of size $n$, let 
$c_n$ be the cardinality of $\C_n$, and let $w_n$ denote the number of Wilf classes formed by the  
permutations in $\C_n$. This allows us at last to define the fundamental concept which we 
will investigate.

\begin{definition}
The class $\C$ has a \emph{Wilf collapse} if $w_n = o(c_n)$ and an \emph{exponential Wilf collapse} 
if, for some $r < 1$, $w_n = o (r^n c_n)$.
\end{definition}

If $\C$ is closed under some symmetry $\phi$ of $\S$ and $\alpha \in \C$ then $\alpha \we_{\C} 
\phi(\alpha)$. However, this never provides a Wilf collapse since $\S$ has only eight symmetries. 
That said, these equivalences will form the core of many of our constructions that do demonstrate 
Wilf collapse.

%

\section{A digression on words}
\label{sec-words}

Let $\A$ be a set of symbols which we will call the \emph{letters} of an \emph{alphabet}. A 
\emph{word} over $\A$ is just a finite sequence (possibly empty) of elements of $\A$ --- the set of 
all words over $\A$ is denoted $\A^{\ast}$ and the set of non-empty words is denoted $\A^+$. The 
empty word is denoted~$\epsilon$. The set $\A^{\ast}$ has an associative operation which is 
normally simply represented by concatenation. 

We generally use lower case letters from near the beginning of the alphabet to denote elements of 
$\A$ and upper case letters from near the end of the alphabet to denote words. That said, we will 
freely identify a letter $a\in\A$ with the corresponding word of length $1$, and thus treat $\A$ as 
a subset of~$\A^+$.

If $W = a_1 a_2 \cdots a_n$ then we say that $a_i$ is the character of \emph{index} $i$. If $W \in 
\A^+$ then $\first(W)$ and $\last(W)$ denote the first and last letter of $W$ respectively; this 
notation is not defined for the empty word, i.e., when used contains an implicit condition that $W 
\neq \epsilon$.

We will further assume that each letter $a\in\A$ has a \emph{weight}, denoted $\weight(a)$, which 
is a positive integer. We extend the weight function to $\A^\ast$ by setting 
\[
\weight(a_1 a_2 \cdots a_k) = \sum_{i=1}^k \weight(a_i).
\]

An \emph{embedding order} is any partial order $\leq$ on $\A^\ast$ satisfying these conditions:
\begin{itemize}
\item For any $W\in\A^\ast$, we have $\epsilon\leq W$.
\item If $W\leq V$ for some $V,W\in\A^\ast$ with $W\neq V$, then $\weight(W)<\weight(V)$.
\item Suppose that $V=a_1a_2\cdots a_k$. Then, for any $W\in \A^*$, we have $W\leq V$ if and only 
if $W$ admits a factorisation $W=W_1W_2\cdots W_k$ such that $W_i\leq a_i$ for each~$i$.
\end{itemize}

A familiar example of an embedding order is the \emph{subword order}, where $W=a_1a_2\cdots a_k$ is 
a subword of $V=b_1b_2\cdots b_{\ell}$ if the sequence $a_1,\dotsc,a_k$ is a (not necessarily 
consecutive) subsequence of $b_1,\dotsc,b_\ell$. In fact, if $W$ is a subword of $V$, then in any 
embedding order $\leq$ we must have $W\leq V$. 

From now on, we assume that $\leq$ is an embedding order for a weighted alphabet~$\A$. A useful 
feature of such orders is that containment can be tested by a natural ``greedy'' procedure, as shown 
by the next proposition. 

\begin{proposition}
\label{prop-greed-is-good}
Let $W, V \in \A^{\ast}$ be two words with $V = b_1 b_2 \cdots b_m$, let $X=b_1\cdots b_i$ be a 
prefix of $V$, and let $Y=b_{i+1}\cdots b_m$ be the corresponding suffix. Let $P$ be the maximal 
prefix of $W$ such that $P \leq X$, and write $W = PS$. Then $W \leq V$ if and only if $S \leq Y$.
\end{proposition}

\begin{proof}
If $P \leq X$ and $S\leq Y$, then $W=PS \leq XY=V$ by the properties of embedding order. 
Conversely, 
if $W \leq V$ then we can write $W = W_1 W_2 \cdots W_m$ with $W_j \leq b_j$. In particular, 
$W_1W_2\cdots W_i$ is either $P$ or a proper prefix of~$P$. Then $S$ is a (not necessarily proper) 
suffix of $W_{i+1}\cdots W_m$ and hence $S \leq Y$.
\end{proof}

\begin{definition}
Let $W,V \in \A^\ast$. If $W \leq V$ and $W \not \leq P$ for any proper prefix $P$ of $V$ then we 
say that $V$ is a \emph{minimal container} for $W$ and write $W \leq^{\ast} V$. Further we define 
generating functions:
\begin{align*}
A(x) &= \sum_{a \in \A} x^{\weight(a)}, \\
I_W(x) &= \sum_{W \leq V} x^{\weight(V)},\\
I_W^\ast(x) &= \sum_{W \leq^\ast V} x^{\weight(V)}.
\end{align*}
\end{definition}

Observe that $W \leq V$ if and only if $V$ can be written as $V = PZ$ where $W \leq^\ast P$ and $Z 
\in \A^\ast$ is arbitrary; moreover $P$ is uniquely determined as the minimal prefix of $V$ which 
is greater than or equal to~$W$. This corresponds to the following identity of generating 
functions.

\begin{observation}
\label{obs-IW-vs-IWstar}
For any $W \in \A^{\ast}$
\[
I_W(x) = \frac{I^{\ast}_{W}(x)}{1 - A(x)}.
\]
In particular $I_W(x) = I_V(x)$ if and only if $I^{\ast}_{W}(x) = I^{\ast}_{V}(x)$. 
\end{observation}

Let us say that two words $W$ and $V$ are \emph{equivalent}, denoted by $W\wwe V$, if 
$I_W(x)=I_V(x)$. Our goal is to show that, under certain assumptions about $\leq$, there are many 
pairs of equivalent words.

\begin{definition}
Let $a,b\in\A$ be two letters. The ordered pair $(a,b)$ is \emph{incompatible} if there is no $c\in 
\A$ such that $ab\leq c$.
\end{definition}

\begin{definition}
Let $W_1, W_2,\dots, W_k \in \A^+$. The factorisation $W = W_1 W_2 \cdots W_k$ is 
\emph{incompatible} if for $1 \leq i < k$, $(\last(W_i), \first(W_{i+1}))$ are incompatible pairs.
\end{definition}

\begin{proposition}
\label{prop-incompatible-factorisation}
Suppose that that $W = W_1 W_2 \cdots W_k$ is an incompatible factorisation. Then
\[
I_W (x) = \frac{\prod_{i=1}^k I^\ast_{W_i}(x)}{1 - A(x)}.
\]
\end{proposition}

\begin{proof}
The proof is by induction on $k$. For $k = 1$ it is simply Observation~\ref{obs-IW-vs-IWstar}. Now 
suppose that $k > 1$ and the result holds for all lesser~$k$. To complete the proof, it suffices to 
show that 
\begin{equation}
I_W(x) = I^{\ast}_{W_1}(x) I_{W_2 W_3 \cdots W_k}(x),\label{eq-w}
\end{equation}
and apply induction. To prove~\eqref{eq-w}, we will show that a word $V\in\A^\ast$ satisfies $W\leq 
V$ if and only if $V$ can be written as $V=XY$ with $W_1\leq^\ast X$ and $W_2 W_3 \cdots W_k\leq 
Y$, and moreover, the $X$ and $Y$ are then determined uniquely. 

Clearly, for any choice of $X$ and $Y$ satisfying $W_1\leq^\ast X$ and $W_2 W_3 \cdots 
W_k\leq Y$, we have $W\leq XY$. To prove the converse, choose $V$ such that $W \leq V$. Then there 
is a unique prefix $X$ of $V$ such that $W_1 \leq^\ast X$. Let $V = XY$, $a = \last(W_1)$, $b = 
\first(W_2)$ and $X = c_1 c_2 \cdots c_n$. By the definition of an embedding order, $W_1$ can be 
written as $W_1=Z_1 Z_2\cdots Z_n$ with $Z_i\leq c_i$. Moreover, from the minimality of $X$, it 
follows that $Z_n\neq\epsilon$ and that $W_1\not\leq c_1\cdots c_{n-1}$. In particular, 
$a=\last(Z_n)\le c_n$. Since $a$ and $b$ are incompatible, we know that $ab\not\leq c_n$, and 
therefore $W_1$ is the longest prefix of $W$ such that $W_1\leq X$. By 
Proposition~\ref{prop-greed-is-good}, we get $W_2\cdots W_k\leq Y$, as claimed.
\end{proof}

The following direct corollary of Proposition 
\ref{prop-incompatible-factorisation} and Observation \ref{obs-IW-vs-IWstar} is 
the keystone in constructing many examples of equivalent words:

\begin{corollary}
\label{cor-incompatible-shuffle}
Suppose that for $1 \leq i \leq k$, $W_i\wwe V_i$, and that $\pi \in \S_k$. If both 
factorisations $W = W_1 W_2 \cdots W_k$ and $V =V_{\pi(1)} V_{\pi(2)} \cdots V_{\pi(k)}$ are 
incompatible then $W\wwe V$.
\end{corollary}

%
%

\subsection{Uniform sampling of words}\label{ssec-uniform}

We will often need to refer to the properties of uniformly random words of a 
given weight in a given set $\A^\ast$. Recall that 
\[
A(x) = \sum_{a \in \A} x^{\weight(a)}
\]
 is the 
generating function of the alphabet $\A$, and define 
\[
A^*(x)=\sum_{W \in \A^\ast} x^{\weight(W)} = \frac{1}{1-A(x)}.
\]

Let $\rho_A$ be the radius of convergence of~$A(x)$. We say that $\A^*$ is \emph{supercritical} if 
\[
\lim_{x\to\rho_A^-} A(x)>1.
\]
If $\A^*$ is supercritical, then the radius of convergence of $A^\ast(x)$ is the unique positive value $\kappa<\rho_A$ such 
that $A(\kappa)=1$. 

We say that $\A^\ast$ is \emph{aperiodic} if the greatest common divisor of 
$\{\weight(a);\; a\in\A\}$ is~1. 

In our setting, where the alphabet $\A$ will generally correspond to the sum-indecomposable elements 
of a permutation class,  aperiodicity is satisfied since there is a letter of weight~$1$. But even 
more generally it is not a significant restriction, since we can simply divide all the weights by 
their greatest common divisor. Supercriticality, on the other hand, is a more fundamental property.

\newcommand{\barw}{\overline{w}}
\newcommand{\wpl}{w^+}
In the rest of Subsection~\ref{ssec-uniform}, we assume that $\A^\ast$ is supercritical, with 
$A(x)$ and $\kappa$ as above, and we fix a probability measure on $\A$ defined by 
$\bbP(a)=\kappa^{\weight(a)}$. This will be the underlying probability measure 
whenever we speak of a random letter from~$\A$.

Let $\barw=\bbE[\weight(a)]$ denote the expected weight of a letter from~$\A$. 
Then
\[
\barw=\sum_{a\in\A} \weight(a)\kappa^{\weight(a)}= \kappa A'(\kappa),
\] 
where $A'$ is the derivative of $A$, and $A'(\kappa)$ is finite since $A$ is analytic at $\kappa$.

As a technical tool, we will use the following concentration inequality, which follows from 
standard probabilistic results; see e.g. the books of Dembo and Zeitouni~\cite[Chapter~2.2]{Dembo} 
or Flajolet and Sedgewick~\cite[Chapter IX.10]{Flajolet}. We include the proof here for 
completeness.


\begin{proposition}\label{pro-con} For every $\varepsilon>0$ there is a $\delta>0$ such 
that for a random word $W=a_1a_2\dotsb a_k$ obtained by concatenating $k$ random independent 
letters from $\A$, we have
\begin{align*}
 \bbP\bigl(\weight(W)\ge (1+\varepsilon)k\barw\bigr)&\le e^{-\delta k} &\text{and}&& 
\bbP\bigl(\weight(W)\le (1-\varepsilon)k\barw\bigr)&\le e^{-\delta k}.
\end{align*}
\end{proposition}
\begin{proof}
Let $X$ be the random variable on $\A$ defined as $X=\weight(a)$ where $a\in \A$ is a random 
letter. Define the function $K(t)=\ln(\bbE[e^{tX}])=\ln\left(\sum_{k\ge 0} 
\bbE[X^k]\frac{t^k}{k!}\right)$. In probability theory, the function $K$ is known as the `cumulant 
generating function'. It can also be written as $K(t)=\ln\left(A(\kappa e^t)\right)$. We may then 
easily check that $K(t)$ is analytic at $t=0$ and has a Taylor series expansion of the form 
$K(t)=\barw t+O(t^2)$ in a neighborhood of~$t=0$.

Let $W=a_1a_2\dotsb a_k$ be a word of $k$ letters chosen independently from~$\A$, and let $X_i$ be 
the weight of $a_i$. In particular, $X_1,\dotsc,X_k$ are independent random variables of the same 
distribution as~$X$. Let $Y=X_1+X_2+\dotsb+X_k$.

Fix an $\varepsilon>0$ and write $\wpl=(1+\varepsilon)\barw$. Let $Q$ denote the event $Y\ge 
k\wpl$, and let $\bbI_Q$ be the indicator function of this event, i.e., the 
function equal to 1 when the event occurs and 0 otherwise. Observe that for any $t\ge 0$, the 
function $\bbI_Q$ is bounded from above by $\exp(tY-tk\wpl)$. We then have
\begin{align*}
\bbP(Y\ge k\wpl)&=\bbE[\bbI_Q]\\
&\le \bbE[\exp(tY-tk\wpl)]\\
&=\exp(-tk\wpl)\bbE[\exp(tY)]\\
&=\exp(-tk\wpl)\bbE\left[\prod_{i=1}^k \exp(tX_i)\right]\\
&=\exp(-tk\wpl)\prod_{i=1}^k \bbE\left[\exp(tX_i)\right]\\
&=\exp(-tk\wpl)\exp(kK(t))\\
&=\exp\left(k(K(t)-t\wpl)\right).
\end{align*}
Recalling that $K(t)=\barw t+O(t^2)$, we can find a sufficiently small value $t>0$ such that 
$K(t)<t\wpl$. Choosing such a $t$ and putting $\delta=t\wpl-K(t)>0$, we obtain $\bbP(Y\ge k\wpl)\le 
e^{-\delta k}$, as claimed. 

The second inequality of the proposition is proven by an analogous argument, except now we consider 
the values $t\le 0$.
\end{proof}

Our main concern will be to understand the structure of random words of fixed weight, which is 
usually much more challenging than dealing with words of fixed length. Let $\A^\ast_n$ be set of 
words of $\A^\ast$ of weight~$n$. Clearly, $\A^\ast_n$ is finite, and we will consider the uniform 
probability measure on this set, i.e., the measure where every word $W\in\A^\ast_n$ has probability
$1/|\A^\ast_n|$. 

To generate such a random word from $W\in\A^\ast_n$, we may use the following process, known as 
\emph{Boltzmann sampler with rejection}. The process works in two phases, where in the first phase, 
it generates a random word of weight at least $n$, and in the next phase, it rejects the generated 
word if its weight is not exactly~$n$.  More precisely, the Boltzmann sampler works as follows.

\textbf{First phase.}
For $i=1,2,\dotsc$, select randomly and independently a letter $a_i\in\A$. Stop as soon as 
$\weight(a_1a_2\dotsb a_i)\ge n$, and let $W=a_1a_2\dotsb a_i$ be the generated word.\\
\textbf{Second phase.} For the word $W$ generated by the first phase, check whether $W$ has 
weight $n$. If it does, the second phase succeeds, and $W$ is output. If not, the second phase 
fails and the whole sampler is restarted.

Samplers of this form were analyzed by Duchon et al.~\cite[Section 7]{DFLS}. Let us summarize the 
main results of their analysis.

\begin{fact}[Duchon et al.~\cite{DFLS}] Suppose that $\A^\ast$ is aperiodic and supercritical. 
Then the second phase of the Boltzmann sampler with rejection succeeds with probability $\Omega(1)$. 
The word output by the sampler is a uniformly random element of $\A^\ast_n$, i.e., each word 
$W\in\A^\ast_n$ is generated with probability $1/{|\A^\ast_n|}$. 
\end{fact}

We will use the Boltzmann sampler to obtain an insight into the structure of a typical word in 
$\A^\ast_n$ as $n$ tends to infinity. For a word $P=p_1p_2\dotsb p_k$, a \emph{$P$-block}
in another word $W=w_1w_2\dotsb w_n$ is a sequence $w_{j+1}w_{j+2}\dotsb w_{j+k}$ of consecutive 
letters such that $p_i=w_{j+i}$ for every $i=1,\dotsc,k$.

\begin{proposition}\label{pro-occur}
For every word $P=p_1p_2\dotsb p_k$ from $\A^\ast$, there is an $\varepsilon\equiv 
\varepsilon_P>0$ such that, with probability at least $1-\frac{1}{2^{\Omega(n)}}$, a uniformly 
random word $X\in\A^\ast_n$ contains at least $\varepsilon n$ pairwise disjoint $P$-blocks.
\end{proposition}
\begin{proof} 
The first phase of the Boltzmann sampler can equivalently be implemented by the 
following procedure: first choose a sequence $V=a_1a_2\dotsb a_n$ of $n$ random independent letters 
from~$\A$, and then output the word $W=a_1a_2\dotsb a_i$ determined as the shortest prefix of $V$ 
of weight at least~$n$. 

Recall that $\barw$ is the expected weight of a letter in~$\A$. Define $m=\left\lfloor 
\frac{n}{2\barw}\right\rfloor$. We will consider two possible `bad' outcomes of the above random 
procedure: the first bad outcome is that the length $i$ of $W$ is smaller than $m$, the second bad 
outcome is that the prefix of $V$ of length $m$ has fewer than $\varepsilon n$ pairwise disjoint 
$P$-blocks, for an $\varepsilon>0$ to be specified later. We will show that both bad outcomes have 
exponentially small probability. Clearly, if neither of the two bad outcomes occurs, then the first 
phase of the Boltzmann sampler generates a word with at least $\varepsilon n$ pairwise disjoint 
$P$-blocks, and since the second phase succeeds with constant probability, this implies that only an 
exponentially small fraction of the words in $\A^\ast_n$ contain fewer than $\varepsilon n$ 
pairwise disjoint $P$-blocks.

The probability of the first bad outcome (i.e., $i<m$) is at most as large as the probability that 
the first $m$ letters of $V$ have weight at least $n\ge 2m\barw$, which is exponentially small by 
Proposition~\ref{pro-con}.

To estimate the probability of the second bad outcome, let $Y$ be the prefix of $V$ of length~$m$. 
Define $q=\left\lfloor \frac{m}{k}\right\rfloor$, and partition $Y$ into subwords as $Y_1Y_2\dotsb 
Y_qZ$, where $Y_1,\dotsc,Y_q$ all have length exactly $k$, and $Z$ is a possibly empty word of 
length at most~$k-1$. The words $Y_1,\dotsc,Y_q$ are pairwise independent, and each of them is a 
random word of length~$k$. In particular, there is a positive probability $\delta>0$ depending on 
$P$ such that for every $j\in[q]$ we have $\bbP(Y_j=P)=\delta$.

There are, therefore, on average $\delta q$ values of $j$ for which $Y_j=P$, and each such value 
corresponds to a $P$-block. By the standard Chernoff--Hoeffding bound~\cite{Hoeffding}, the 
probability that there are fewer than $\delta q/2$ values of $j$ satisfying $Y_j=P$ is exponentially 
small in $q$, and therefore also in~$n$.

We conclude that the word output by the sampler, which is a uniformly random word from $\A^\ast_n$, 
contains at least $\delta q/2$ disjoint $P$-blocks, up to exceptions of exponentially small 
probability. Since $q = \Omega(n)$ this completes the proof.
\end{proof}

\section{Wilf collapse}
\label{sec-wilf}

In this section, we will present our two main results demonstrating Wilf collapse in a permutation 
class $\C$ under different sets of assumptions. The two results deal with sum-closed classes and 
with classes with finitely many sum-indecomposables, respectively. Their proofs all follow the same 
general strategy:
\begin{itemize}
\item
Represent the elements of $\C$ as words over the alphabet consisting of the sum-indecomposable 
permutations in $\C$.
\item
Using Corollary \ref{cor-incompatible-shuffle} characterise some ``good'' elements of $\C$ whose 
equivalence classes with respect to $\we_{\C}$ are ``large'',
 \item
Show that permutations in $\C$ are bad with ``sufficiently small'' probability.
\end{itemize}

The number of Wilf classes for $\C$ among elements of size $n$ is bounded above by the sum of the 
number of good permutations in $\C_n$ divided by the smallest size of a good Wilf class, and the 
number of bad permutations in $\C_n$. Therefore, the scheme above is sufficient to prove a Wilf 
collapse provided that ``large'' implies tending to infinity, and ``sufficiently small'' means 
tending to 0. To obtain an exponential Wilf collapse it is sufficient that ``large'' should mean 
``of exponential size'' and that the probability of a permutation being bad is exponentially small.

\subsection{Sum-closed classes}
\label{ssec-sumcl}

Let $\C$ be a sum-closed class. Take the alphabet $\A$ to consist of the sum-indecomposable 
permutations of $\C$ with the weight of a letter simply being equal to its size. Then we already 
have an obvious bijection between $\C$ and $\A^\ast$ which we now treat as implicit, i.e., we make 
no distinction between a permutation in $\C$ and its representation as (the sum of) a sequence of 
sum-indecomposable permutations. In particular, we say that $\C$ is supercritical whenever 
$\A^\ast$ is. We also extend the containment order on $\C$ to words of $\A^\ast$; that is, for 
$W,V\in\A^\ast$ we write $W\leq V$ if the permutation represented by $W$ is contained in the 
permutation represented by~$V$. Observe that this partial order on $\A^\ast$ is an 
embedding order.

\begin{theorem}
\label{thm-sumclosed}
Any supercritical sum-closed class, $\C$,  that contains an incompatible pair has an exponential Wilf collapse, unless $\C$ is the class of increasing permutations.
\end{theorem}

\begin{proof}
Let $\C$ be a supercritical sum-closed class containing an incompatible pair $(a,b)$ and not equal to the class of increasing permutations. The set, $\A$, of sum-indecomposable permutations in $\C$ has at least the two elements $1$ 
and~$21$. Let $c$ and $d$ be two arbitrary distinct 
elements from~$\A$. Consider the words $X=bca$, $Y=bda$, $P=aXYb=abcabdab$, and $P'=aYXb=abdabcab$. 
Notice that both $P=aXYb$ and $P'=aYXb$ are incompatible factorisations.

By Proposition~\ref{pro-occur}, there is an $\varepsilon>0$ such that for every $n$, a uniformly 
random permutation $\pi\in\C$ of order $n$ has a sum decomposition in which there are at least 
$\varepsilon n$ disjoint $P$-blocks, except for an exponentially small fraction of `bad' 
permutations. By Corollary~\ref{cor-incompatible-shuffle}, if a permutation $\pi'$ is obtained from 
$\pi$ by replacing some $P$-blocks by $P'$-blocks, then $\pi$ and $\pi'$ are Wilf-equivalent. In 
particular, the Wilf class of a permutation that has at least $\varepsilon n$ disjoint $P$-blocks 
has size at least $2^{\varepsilon n}$. The theorem follows.
\end{proof}

\begin{corollary}
\label{cor-sumclosedfinitealphabet}
If $\C$ is a sum-closed class that contains 21 and has only finitely many sum-indecomposable 
permutations, then $\C$ has an exponential Wilf collapse.
\end{corollary}

\begin{proof}
Such a class is clearly supercritical. It also has an incompatible pair, e.g., 1 and any 
sum-indecomposable permutation of maximum size, so Theorem~\ref{thm-sumclosed} applies.
\end{proof}

\subsection{Classes with finitely many sum-indecomposables}\label{ssec-finin}

The aim of this subsection is to prove the following result.
\begin{theorem}\label{thm-finind}
Any permutation class with finitely many sum-indecomposable permutations and an unbounded growth 
function exhibits a Wilf collapse.
\end{theorem}

The proof of this result is rather technical so we will begin with a few words about its general 
strategy. Let $\C$ be a permutation class with finitely many indecomposables and an unbounded growth 
function. As in Subsection~\ref{ssec-sumcl}, we will represent the elements of $\C$ as words over 
the alphabet $\A$ of indecomposable elements of~$\C$. However, not all words from $\A^\ast$ now 
correspond to elements of~$\C$, so we cannot directly use the properties of $\A^\ast$ to prove the 
Wilf collapse of~$\C$. Instead, we consider a finite state automaton over $\A$ that accepts 
only the words which represent elements of~$\C$. The underlying graph of this automaton is directed 
and acyclic except for loops on certain states. These loops represent certain sum-closed subclasses 
of $\C$ whose elements occur as consecutive subwords (``loop blocks'') within the elements of~$\C$. 

Since the classes corresponding to loop blocks are sum-closed, 
Corollary~\ref{cor-sumclosedfinitealphabet} applies to them, and so, unless the only symbol that 
allows for a loop at a given state is $1$, they have exponential Wilf collapse. We will then show 
that under suitable technical assumptions, which a random permutation of $\C$ satisfies with high 
probability, the Wilf equivalences within the class generated by a loop block can be lifted to Wilf 
equivalences for the whole class~$\C$.

We begin with a general lemma dealing with the growth rate of a set of words generated by a finite 
alphabet.

\begin{lemma}\label{lem-grow} Let $\L$ be a finite alphabet, with every letter $a\in\L$ having a 
positive integer weight $\weight(a)$. Let $k$ be the largest weight of a letter of~$\L$, and let 
$\alpha_i$ be the number of letters in $\L$ of weight $i$, for $i=1,\dotsc,k$. Assume that 
$\alpha_1=1$, i.e., there is a unique letter of weight 1. Let $L(x)=\sum_{i=1}^k \alpha_ix^i$ be the 
generating polynomial of~$\L$. 

The polynomial $1-L(x)$ has a unique positive real root $\rho$, this root belongs to the interval 
$(0,1]$, has multiplicity 1, and any other complex root $\lambda$ of $1-L(x)$ satisfies 
$|\lambda|>\rho$. Moreover, $\rho$ is equal to $1$ if and only if $|\L|=1$. There are constants 
$c>0$ and $\varepsilon>0$ such that $|\L^\ast_n| =c\rho^{-n} + O( (\rho+\varepsilon)^{-n})$.
\end{lemma}
\begin{proof}
Noting that $\alpha_1=1$ and $\alpha_2,\dotsc,\alpha_k$ are all nonnegative, we observe that 
$1-L(0)=1$ and $1-L(1)\le 0$, and therefore there is a $\rho\in(0,1]$ such that $1-L(\rho)=0$. We 
also see that $\rho=1$ if and only if $L(x)=x$, or equivalently, $|\L|=1$. 
Since the derivative of $1-L(x)$ is negative for every $x>0$, we conclude that 
$\rho$ is the unique positive root of $1-L(x)$, and that it has multiplicity 1. 

Suppose now that $\lambda$ is a complex root of $1-L(x)$, with $|\lambda|\le\rho$. We claim 
that $\lambda=\rho$. Let $\Re(z)$ denote the real part of a complex number~$z$. We then obtain
\[
1=\sum_{i=1}^k \alpha_i\lambda^i=\sum_{i=1}^k \alpha_i\Re(\lambda^i)\le \sum_{i=1}^k
\alpha_i|\lambda|^i\le\sum_{i=1}^k \alpha_i\rho^i=1.
\]
In particular, all the inequalities hold with equality. Since for each $i\in[k]$, we have 
$\Re(\lambda^i)\le |\lambda|^i\le\rho^i$, and since $\alpha_1>0$, we easily deduce that 
$\lambda=\rho$.

The generating function of $\L^\ast$ is $L^\ast(x)=1/(1-L(x))$. We see that $L(x)$ is a rational 
function that has a simple pole at $x=\rho$, and any other pole has absolute value greater than 
$\rho+\varepsilon$, for some $\varepsilon>0$. From this, the asymptotics of $|\L^\ast_n|$ follow, 
by standard singularity analysis~\cite{Flajolet}.
\end{proof}

We call the value $\rho^{-1}$ from the previous lemma the \emph{growth rate} of $\L^\ast$.

Suppose from now on that $\C$ is a permutation class with finitely many sum-indecomposable 
permutations, and that its growth function is unbounded. Let $\A$ be the set of sum-indecomposable 
permutations in~$\C$. By the above assumptions, $\A$ is finite and contains at least two distinct 
elements, namely $1$ and $21$. Let $K$ denote, from now on, the largest weight of a letter of~$\A$.

As in Subsection~\ref{ssec-sumcl}, we will represent the permutations in $\C$ as words over the 
alphabet~$\A$, and assign to each letter of $\A$ the weight equal to the size of the corresponding 
sum-indecomposable permutation. 
Since Corollary \ref{cor-sumclosedfinitealphabet} deals with the sum-closed case, suppose 
from now on that $\C$ is not sum-closed, i.e., not every word in $\A^\ast$ corresponds to a 
permutation from~$\C$. 

Let $F$ be the set of minimal elements of $\A^\ast$ not belonging to $\C$; in particular, we have 
$\C=\A^\ast\cap\Av(F)$. By the classical Higman Lemma~\cite{Higman}, the set $F$ is finite. Let us 
write $f=|F|$ and $F=\{\phi^1,\phi^2,\dotsc,\phi^f\}$. The forbidden permutations $\phi^i$ will 
again be interpreted as words over~$\A$. We let $\ell_i$ denote the number of symbols of the 
word~$\phi^i$.

For a word $W=a_1a_2\dotsb a_m\in\A^\ast$, let $W_{\le i}$ denote its prefix $a_1a_2\dotsb 
a_i$, and $W_{\ge i}$ its suffix $a_ia_{i+1}\dotsb a_m$. We will also use the notation $W_{<i}$ and 
$W_{>i}$ for $W_{\le i-1}$ and $W_{\ge i+1}$, respectively.

Let $W\in\A^\ast$ be a word. The \emph{prefix state} of $W$ is the $f$-tuple 
$(p_1,p_2,\allowbreak\dotsc,p_f)$ where $p_i$
is the length of the longest prefix of 
$\phi^i$ that is contained in~$W$, or in other words, for every $i=1,\dotsc,f$, the word
$W$ contains $\phi^i_{\le p_i}$ but avoids~$\phi^i_{\le p_i+1}$. 

We say that the prefix state $p=(p_1,\dotsc,p_f)$ is \emph{valid}, if $p_i<\ell_i$ for each~$i$. 
Observe that a word $W\in\A^\ast$ is in $\C$ if and only if its prefix state is valid. Let $\P$ be 
the set of all possible prefix states of the elements of~$\C$. The empty word has prefix state 
$(0,0,\dotsc,0)$, which we will call \emph{the initial prefix state}.

For two prefix states $p=(p_1,\dotsc,p_f)$ and $p'=(p'_1,\dotsc,p'_f)$, we write $p\le p'$ if 
$p_i\le p'_i$ for every $i\in[f]$, and we write $p<p'$ if $p\le p'$ and~$p\neq p'$.

Let $X$ be a word with a prefix state $p$, and let $a\in\A$ be a symbol. The prefix state of 
the word $Y=Xa$ is then uniquely determined by $p$ and~$a$. Moreover, if $p'$ the prefix state of 
$Y$, then $p\le p'$. If $p=p'$, we say that the symbol $a$ is a \emph{loop symbol for $p$}, 
otherwise we say that $a$ is a \emph{transition symbol from $p$ to~$p'$}. The \emph{loop alphabet of 
$p$}, denoted $\L_p$, is the set of the loop symbols of~$p$. 

A symbol $a\in\A$ is in the loop alphabet of a valid prefix state $p=(p_1,\dotsc,p_f)$ if 
and only if for each $i=1,\dotsc,f$, the $(p_i+1)$-st symbol of $\phi^i$ is not contained in~$a$. 
In particular, the loop alphabet of $p$ is a down-set of~$\A$. 

Consider a word $W=a_1a_2\dotsb a_m$ from $\C$. We will say that $W$ has a \emph{prefix transition} 
at position $i$ if $W_{<i}$ has a different prefix state than~$W_{\le i}$. Let $k$ be the number of 
prefix transitions in $W$, and let $i(1)<i(2)<\dotsb <i(k)$ be the positions where the transitions 
occur. We call the sequence of prefix states $p(0)<p(1)<\dotsb<p(k)$ such that the transition at 
position $i(j)$ is from $p(j-1)$ to $p(j)$ the \emph{prefix transition path} of~$W$. The words in 
$\C$ determine only finitely many possible prefix transition paths. The word $W$ can 
then be written as 
\begin{equation}
 W=B_0 a_{i(1)} B_1 a_{i(2)} B_2\dotsb B_{k-1} a_{i(k)} B_k,
\label{eq-decomp}\end{equation}
where $B_j$ is a (possibly empty) word over the loop alphabet $\L_{p(j)}$. We call $B_j$ the 
\emph{$j$-th loop block} of $W$, and we call the right-hand side of~\eqref{eq-decomp} the 
\emph{prefix decomposition} of~$W$. 

Conversely, suppose that $p(0)<p(1)<\dotsb<p(k)$ is an increasing sequence of valid prefix states 
where $p(0)$ is the initial state, that $t_j\in\A$ is a transition symbol from $p(j-1)$ to $p(j)$, 
and that $B_j$ is a possibly empty word over the alphabet $\L_{p(j)}$. Then the expression 
\[
 B_0 t_1 B_1 t_2 B_2\dotsb B_{k-1} t_k B_k
\]
is the prefix decomposition of a word from~$\C$ whose prefix transition path is  $p(0)<p(1)<\dotsb<p(k)$.

For a prefix state $p\in\P$, the \emph{growth rate of $p$}, denoted $\gamma_p$, is the growth 
rate of the language~$\L_p^\ast$. The \emph{dominant growth rate} of $\C$ is the value 
$\gamma=\max_{p\in\P} \gamma_p$. We say that a prefix state $p$ is \emph{dominant} if 
$\gamma_p=\gamma$, and we call a loop block in a prefix decomposition \emph{dominant} if it 
corresponds to a dominant state. Let $D$ denote the largest number of dominant states that can 
appear on a single prefix transition path of a word from~$\C$. 


\begin{proposition}\label{pro-growth}  Let $T=(p(0)<p(1)<\dotsb<p(k))$ be a prefix transition path, 
and let $\tau=(t_1,\dotsc,t_k)\in\A^k$ be a sequence of symbols where $t_j$ is a transition symbol 
from $p(j-1)$ to~$p(j)$. Let $\C_n(T,\tau)$ be the set of words in $\C_n$ that have prefix 
transition path $T$, with $j$-th transition on the symbol~$t_j$. Let $\gamma_T$ be the 
maximum of $\gamma_{p(i)}$ for $i\in\{0,\dotsc,k\}$, and let $d_T$ be the number of values 
$i\in\{0,\dotsc,k\}$ for which $\gamma_{p(i)}=\gamma_T$.

Then 
$
|\C_n(T,\tau)|=\Theta\left(n^{d_T-1}\gamma_T^n\right).
$
Consequently,
$
 |\C_n|= \Theta\left(n^{D-1}\gamma^n\right).
$
\end{proposition}
\begin{proof}
Let $n' = n-\sum_{i=1}^k\weight(t_i)$. Let $\gamma_j$ be the growth rate of~$p(j)$. To count the 
words $W\in\C_n(T,\tau)$, we will count their corresponding prefix decompositions 
$W=B_0t_1 B_1\dotsb B_{k-1}t_k B_k$, or equivalently, the $(k+1)$-tuples $(B_0,B_1,\dotsc,B_k)$ 
with $\sum_{j=0}^k\weight(B_j)=n'$, where $B_j$ is a word over the alphabet $\L_{p(j)}$.

                                                                                                 
Define the sets of indices  $I=\{i\in\{0,\dotsc,k\};\; \gamma_i=\gamma_T\}$ and $J=\{0,\dotsc, 
k\}\setminus I$. In particular, $|I|=d_T$. To estimate the number of prefix decompositions of 
elements of $\C_n(T,\tau)$, we will first fix two integers $n_I$ and $n_J$ with $n_I+n_J=n'$, and 
then count the decompositions in which the loop blocks with growth rate $\gamma_T$ have total weight 
$n_I$, and the remaining loop blocks have total weight~$n_J$. 

To count the possible choices for $(B_i;\; i\in I)$ with $\sum_{i\in I} 
\weight(B_i)=n_I$, we first choose a $d_T$-tuple $(n_i;\; i\in I)$ satisfying $\sum_{i\in I} 
n_i=n_I$, and then choose $B_i\in \L^\ast_{p(i)}$ of weight~$n_i$. The number of the suitable 
$d_T$-tuples $(n_i;\; i\in I)$ is $\Theta(n_I^{d_T-1})$, and for any $i\in I$, there are 
$\Theta(\gamma_T^{n_i})$ choices for~$B_i$ by Lemma~\ref{lem-grow}. Overall, the number of possible 
choices of $(B_i;\; i\in I)$ for a fixed $n_I$ is $\Theta(n_I^{d_T-1}\gamma_T^{n_I})$.

To count the choices of $(B_j;\; j\in J)$ satisfying $\sum_{j\in J} \weight(B_j)=n_J$, let 
$\delta>0$ be a value smaller than $\gamma_T$ but larger than $\gamma_j$ for any~$j\in J$. We then 
have $\Theta(n_J^{|J|-1})$ ways to choose a $|J|$-tuple $(n_j;j\in J)$ with $\sum_{j\in J} 
n_j=n_J$, and for each $j\in J$, $\Theta(\gamma_j^{n_j})$ ways to choose a block $B_j\in 
\L_{p(j)}^\ast$ of weight~$n_j$. The number of choices for $(B_j;\; j\in J)$ is thus 
$\cO(\delta^{n_J})$.

This yields
\[
 |\C_n(T,\tau)|= \sum_{(n_I,n_J)} \cO(\delta^{n_J})\Theta\left( n_I^{d_T-1}\gamma_T^{n_I}\right),
\]
where the summation is over all pairs $(n_I,n_J)$ satisfying $n_I+n_J=n'$. We see that the summand 
corresponding to $n_J=0$ in this sum already has order $\Theta(n^{d_T-1}\gamma^n)$, so we only need 
an upper bound for $|\C_n(T,\tau)|$ of the same order. Such an upper bound can be 
obtained as follows:
\begin{align*}
|\C_n(T,\tau)|&= \sum_{(n_I,n_J)} \cO(\delta^{n_J})\Theta\left( n_I^{d_T-1}\gamma_T^{n_I}\right)\\
&\le \Theta\left(n^{d_T-1}\gamma_T^n\right) \sum_{(n_I,n_J)} \frac{\delta^{n_J}}{\gamma_T^{n_J}}\\
&\le\Theta\left(n^{d_T-1}\gamma_T^n\right) 
\sum_{n_J=0}^{\infty}\left(\frac{\delta}{\gamma_T}\right)^{n_J}\\
&\le\Theta\left(n^{d_T-1}\gamma_T^n\right).
\end{align*}
The bound for $|\C_n|$ then follows by summing $|\C_n(T,\tau)|$ over all possible (finitely 
many) choices of $T$ and $\tau$, noting that these choices are independent of~$n$.
\end{proof}

\begin{lemma}\label{lem-nondom}
Let $W$ be a uniformly random word from the set $\C_n$. With probability $1-2^{-\Omega(\sqrt n)}$, 
the total weight of the non-dominant loop blocks in the prefix decomposition of $W$ is smaller than 
$\sqrt n$. 
\end{lemma}
\begin{proof} Proposition~\ref{pro-growth} shows that the fraction of words from $\C_n$ 
whose prefix transition path $T$ has no dominant state is exponentially small, so let us focus 
on words with at least one dominant loop state. Applying the argument and the notation of the proof 
of Proposition~\ref{pro-growth} to such $T$, we have $\gamma_T=\gamma$, $d_T\le D$, and the 
number of words in $\C_n(T,\tau)$ whose non-dominant loop blocks have total weight at least $\sqrt 
n$ is at most
\begin{align*}
\sum_{n_J=\sqrt n}^{n}\! \cO\left(\delta^{n_J} n^{D-1}\gamma^{n-n_J}\right)&\le 
\cO\left(n^{D-1}\gamma^n\right) 
\sum_{n_J=\sqrt n}^{\infty}\left(\frac{\delta}{\gamma}\right)^{n_J}
\le \cO\left(n^{D-1}\gamma^n \left(\frac{\delta}{\gamma}\right)^{\!\sqrt n}\right).
\end{align*}
The lemma follows.
\end{proof}

Recall that $F=\{\phi^1,\dotsc,\phi^f\}$ is the set of minimal elements of $\A^\ast$ not belonging 
to~$\C$, and that $\ell_j$ is the length of $\phi^j$. To proceed with our argument, we now need to 
also start considering suffix states of a word, which are analogous to prefix states. The 
\emph{suffix state} of a word $W\in\C$ is the $f$-tuple $(s_1,\dotsc,s_f)$ where $s_j$ is the length 
of the longest suffix of $\phi^j$ contained in~$W$. We say that a suffix state $s=(s_1,\dotsc,s_f)$ 
is \emph{valid} if $s_j<\ell_j$ for each $j$, and we let $\S$ be the set of all the possible suffix 
states of the words from~$\C$ (which are necessarily valid).

In analogy with prefix states, we can associate to a suffix state $s\in\S$ a \emph{suffix loop 
alphabet} $\coL_s$, which is the set of all the symbols $a\in\A$ such that if $W$ has suffix state 
$s$ then $aW$ has suffix state $s$ as well. We also say that a symbol $b$ is a suffix transition 
symbol from $s$ to $s'$, if for a word $W$ of suffix state $s$, the word $bW$ has suffix state~$s'$.

We can also define suffix decompositions, analogous to their prefix counterparts, but obtained when 
scanning a word from right to left. We say that a word $W=a_1a_2\dotsb a_m\in\C$ has a \emph{suffix 
transition} at position $i$ if $W_{\ge i}$ has a different suffix state than $W_{>i}$. Let 
$i(1)>i(2)>\dotsb>i(k)$ be all the positions where $W$ has a suffix transition, ordered right to 
left, and let $s(0)<s(1)<\dotsb<s(k)$ be the corresponding suffix transition path; that is, 
$W_{>i(j)}$ has suffix state $s(j-1)$ and $W_{\ge i(j)}$ has suffix state $s(j)$. The 
\emph{suffix decomposition} of $W$ then takes the form $W= B_ka_{i(k)}B_{k-1}\dotsb 
B_1 a_{i(1)} B_0$, where $B_j$ is a word over the suffix loop alphabet of $s(j)$.

Suffix decompositions satisfy analogous properties as their prefix counterparts. For instance, 
Proposition~\ref{pro-growth} and Lemma~\ref{lem-nondom} remain true when restated in the setting of 
suffix decompositions. This implies, in particular, that the largest growth rate of a suffix state 
is the same as the largest growth rate $\gamma$ of a prefix state, and the maximum number of 
dominant blocks in a suffix decomposition is the same as the maximum number $D$ of dominant blocks 
in a prefix decomposition.

Let $p=(p_1,\dotsc,p_f)$ be a prefix state of a word $X$, and let $s=(s_1,\dotsc,s_f)$ be a 
suffix state of a word~$Y$. Recall that $\ell_i$ is the length of the forbidden word $\phi^i\in
F$. We say that the two states $p$ and $s$ \emph{overlap}, if for some $i\in[f]$, we have 
$p_i+s_i\ge \ell_i$. The inequality $p_i+s_i\ge \ell_i$ holds if and only if $XY$ 
contains~$\phi^i$. In particular, $p$ and $s$ overlap if and only if $XY$ is not in~$\C$. 
An invalid prefix state overlaps with any suffix state and vice versa. If $p$ and $s$ do not 
overlap, we say that they are \emph{compatible}.

\begin{lemma}\label{lem-xydom}
Let $X$ be a word with $d_X$ dominant loop blocks in its prefix decomposition, and let $Y$ be a word 
with $d_Y$ dominant loop blocks in its suffix decomposition. If $XY$ is in $\C$, then $d_X+d_Y\le 
D+1$. 
\end{lemma}
\begin{proof}
If $d_X=0$ or $d_Y=0$, the claim follows trivially, so assume that $d_X$ and $d_Y$ are both 
positive.

Let $T_X$ be the prefix transition path of $X$ and $T_Y$ the suffix transition path of~$Y$. Let 
$\C_n(T_X)$ be the subset of $\C_n$ of those words that have prefix transition path $T_X$, and 
$\C^\suf_n(T_Y)$ be the subset of $\C_n$ of words whose suffix transition path is~$T_Y$. In 
particular, any word in $\C_n(T_X)$ has prefix state $p$, while words in $\C^\suf_n(T_Y)$ have 
suffix state~$s$. This means that a concatenation of a word from $\C_n(T_X)$ with a word from 
$\C^\suf_n(T_Y)$ yields a word from $\C_{2n}$. With the help of Proposition~\ref{pro-growth}, we get
\[
 |\C_{2n}|\ge 
|\C_n(T_X)|\cdot|\C^\suf_n(T_Y)|=\Theta(n^{d_X-1}\gamma^n)\cdot\Theta(n^{d_Y-1}\gamma^n)=\Theta(n^{
d_X+d_Y-2}\gamma^{2n}).
\]
On the other hand, we know that $|\C_{2n}|=\Theta(n^{D-1}\gamma^{2n})$. The lemma follows.
\end{proof}

Let us say that a prefix state $p'\in\P$ is a \emph{prefix successor} of $p\in\P$ if $p'\neq p$ 
and there is a symbol $b\in\A$ which is a transition symbol from $p$ to $p'$. Suffix successors are 
defined analogously.

For a prefix state $p\in \P$ and a suffix state $s\in\S$, we say that $p$ and $s$ \emph{match}, if 
they are compatible, but every prefix successor of $p$ overlaps with $s$, and every suffix 
successor of $s$ overlaps with~$p$. 

As an example, consider $\A=\{a,b,c,d\}$ and $F=\{\phi^1,\phi^2\}$ with $\phi^1=abc$ and 
$\phi^2=dbdbc$. Let $X=da$ and $Y=c$. Then $X$ has prefix state $p=(1,1)$ and $Y$ has 
suffix state $s=(1,1)$. The two states match: the only prefix successor of $p$ is the state $(2,2)$ 
which overlaps $s$, and the only suffix successor of $s$, namely $(2,2)$, overlaps~$p$. Consider 
now the word $X'=dba$: its prefix state $p'=(1,2)$ is also compatible with $s$, but it does not 
match with $s$, since it has the successor $(1,3)$ compatible with~$s$. Notice that, perhaps 
non-intuitively, although $p$ and $s$ match and $p<p'$ (in fact $X< X'$), the two compatible states 
$p'$ and $s$ do not match.

\begin{lemma}\label{lem-match1}
If $p\in\P$ and $s\in\S$ match, then $\L_p=\coL_s$. 
\end{lemma}
\begin{proof}
Suppose for contradiction that $\L_p\neq\coL_s$. Assume, without loss of generality, that there is 
a symbol $b\in\L_p\setminus\coL_s$. Consider a word $X$ with prefix state $p$ and a word $Y$ with 
suffix state~$s$. Since $p$ and $s$ are compatible, $XY$ is in~$\C$. Note that since $b$ is in 
$\L_p$, $Xb$ has prefix state~$p$. Let $s'$ be the suffix state of~$bY$. Then $s'$ is a suffix 
successor of $s$, and therefore it overlaps with~$p$. Then the word $W=XbY$ is in $\C$, since the 
prefix state of $Xb$ is compatible with the suffix state of $Y$; on the other hand, $W$ is not in 
$\C$, since the suffix state of $bY$ overlaps the prefix state of $X$. This is a contradiction.
\end{proof}

\begin{lemma}\label{lem-match2}
Let $p\in\P$ and $s\in\S$ be a pair of compatible states with $\L_p=\coL_s$. Then the following 
are equivalent:
\begin{enumerate}
\item[(I)] The states $p$ and $s$ do not match.
\item[(II)] The state $p$ has a prefix successor compatible with $s$.
\item[(III)] The state $s$ has a suffix successor compatible with $p$.
\item[(IV)] There is a symbol $b\in\A\setminus\L_p$ such that for any word $X$ with prefix state $p$ 
and any word $Y$ with suffix state $s$, the word $XbY$ is in~$\C$.
\end{enumerate}
\end{lemma}
\begin{proof}
If (IV) holds, then the prefix state of $Xb$ is a prefix successor of $p$ compatible with $s$, 
while the suffix state of $bY$ is a suffix successor of $s$ compatible with $p$, so (II) and (III) 
hold as well, and clearly both (II) and (III) implies (I). We also easily see that (II) implies 
(IV) and (III) implies (IV), and therefore (II), (III) and (IV) are equivalent. Finally, (I) 
implies that (II) or (III) holds, and therefore (I) implies (IV), completing the proof.
\end{proof}

\begin{lemma}\label{lem-tight}
Let $X$ be a word with prefix state $p$ and with $d_X$ dominant blocks in its prefix decomposition, 
and $Y$ a word with suffix state $s$ and $d_Y$ dominant blocks in its suffix decomposition. Suppose 
that $p$ and $s$ are compatible, that $d_X+d_Y=D+1$, that $\L_p=\coL_s$, and that the state $p$ (and 
therefore also $s$) is dominant. Then the two states $p$ and $s$ match.
\end{lemma}
\begin{proof}
Suppose that $p$ and $s$ do not match. Then, by part (IV) of Lemma~\ref{lem-match2}, there is a 
symbol $b\in\A\setminus\L_p$ such that $XbY$ is in~$\C$. Let $U$ be the word obtained by 
concatenating all the symbols of $\L_p$ in any order, and let $W$ be the concatenation of $|\P|$ 
disjoint copies of~$U$. Consider the word $Z=XbWY$. Since $W$ is a word over the alphabet 
$\L_p$, which is equal to $\coL_s$, we know that $WY$ has the same suffix state as $Y$, namely~$s$. 
Since the prefix state of $Xb$ is compatible with $s$ (recall that $XbY\in\C$), we conclude that 
$XbWY$ is in~$\C$. 

We claim that the prefix decomposition of $XbW$ has more dominant blocks than the prefix 
decomposition of $X$, i.e., $XbW$ has at least $d_X+1$ dominant blocks. To see this, note that the 
prefix decomposition of a word in $\C$ has at most $|\P|-1$ transitions. This means that in 
$W$, there is a copy of $U$ which does not contain any prefix transition. This copy of $U$ is thus 
entirely contained in a single loop block $B$, whose corresponding loop alphabet therefore 
contains all the symbols of $\L_p$. Since $\L_p$ is dominant by assumption, the loop block $B$ is 
also dominant. Moreover, $B$ is entirely contained in $W$, since $b$ is a transition symbol by 
construction. Thus, $XbW$ has at least $d_X+1$ dominant blocks. This contradicts 
Lemma~\ref{lem-xydom}, applied to the decomposition of $Z$ into $XbW$ and~$Y$.
\end{proof}

We remark that the assumptions of Lemma~\ref{lem-tight} are actually redundant: the lemma remains 
true even without assuming that $\L_p=\coL_s$ and that $p$ and $s$ are dominant. In fact, these two 
assumptions are themselves consequences of $d_X+d_Y=D+1$. However, we will not need this stronger 
fact.

Fix now the constant $Q = 3(|\P|+|\S|)$. The choice of $Q$ guarantees that whenever a word 
$W\in\C$ is expressed as a concatenation of $Q$ subwords $W=W_1W_2\dotsb W_Q$, there 
will be three consecutive subwords $W_{i-1},W_i,W_{i+1}$ such that none of them contains a 
prefix transition or a suffix transition of~$W$; this is because each $W\in\C$ has at most 
$|\P|-1$ prefix transitions and at most $|\S|-1$ suffix transitions.

An \emph{equitable partition} of a word $W\in\C_n$ is the expression $W=W_1W_2\dotsb W_Q$, where 
the $W_j$ are chosen in such a way that for every $j\in[Q]$, the prefix $W_1W_2\dotsb W_j$ is the 
shortest prefix of $W$ whose weight is at least $jn/Q$. We call $W_j$ the \emph{$j$-th slice} 
of the equitable partition. Recall that $K$ is the largest weight of a symbol in $\A$, and note that 
the above definition guarantees that $jn/Q\le\weight(W_1\dotsb W_j)< jn/Q+K$ for every $j\in[Q]$. 
In particular, each $W_j$ satisfies $n/Q-K<\weight(W_j)<n/Q+K$.

A slice $W_j$ in the equitable partition of a word $W\in\C_n$ is \emph{free} if it does 
not contain any prefix transition or suffix transition of~$W$. This means that in the 
prefix decomposition of $W$, as well as in the suffix decomposition of $W$, the free slice belongs 
to a single loop block. Our choice of $Q$ guarantees that the equitable partition of any word 
$W\in\C$ contains three consecutive free slices. Let $\C_n(j)$ be the subset of $\C_n$ containing 
the words whose $j$-th slice is free. 

Recall that $D$ is the largest number of dominant loop blocks in a prefix decomposition of a word 
from $W$, which is also equal to the largest number of dominant loop blocks in a suffix 
decomposition of a word in~$W$. We say that a word $W\in\C_n$ is \emph{typical}, if for every free 
slice $W_j$ in the equitable partition of~$W$, the following conditions hold:
\begin{enumerate}
\item $W_j$ is contained in a dominant loop block both in the prefix decomposition of $W$ and in 
the suffix decomposition of $W$.
\item Let $X_j=W_1W_2\dotsb W_{j-1}$ and $Y_j=W_{j+1}W_{j+2}\dotsb W_Q$. Let $d_X$ be the number 
of dominant loop blocks in the prefix decomposition of $X_j$, and let $d_Y$ be the number of 
dominant loop blocks in the suffix decomposition of~$Y_j$. Then $d_X+d_Y=D+1$.
\item With $X_j$ and $Y_j$ as above, let $p_j$ be the prefix state of $X_j$ and let $s_j$ be the 
suffix state of $Y_j$. Then $p_j$ and $s_j$ match.
\item Let $\L_j$ be the prefix loop alphabet of the state $p_j$. 
Then for any word 
$P$ of length $\max\{4,|\L_j|\}$ over the alphabet $\L_j$, $W_j$ contains at least $\sqrt n$ 
disjoint $P$-blocks.
\end{enumerate}

\begin{lemma}\label{lem-typ2}
A uniformly random word $W\in\C_n$ is typical with probability $1-\cO(1/n)$. 
\end{lemma}
\begin{proof}
Lemma~\ref{lem-nondom} shows that with probability $1-2^{-\Omega(\sqrt n)}$, the total weight of 
non-dominant 
loop blocks in $W$ is at most $\sqrt n$, and in particular, for $n$ large enough, no slice can be 
contained in a single non-dominant loop block. Therefore, any free slice is contained in a dominant 
loop block, both in the prefix and in the suffix decomposition. 

Let $W_j$ be a free slice, and let $X_j$, $Y_j$, $p_j$ and $s_j$ be as in the definition of 
typical word. Let $\L_j$ be the prefix loop alphabet of $p_j$, and let $\coL_j$ be the suffix loop 
alphabet of~$s_j$. Define $\L=\L_j\cap\coL_j$. Since $W_j$ contains no prefix or suffix transitions, 
all its symbols belong to~$\L$. 

Let us prove that with high probability, $\L_j=\coL_j$. Suppose that this is not the case, and 
without loss of generality assume that $\L$ is a proper subset of~$\L_j$. As we have seen, $p_j$ 
and $s_j$ are (with high probability) dominant states, so they both have growth rate~$\gamma$. Since 
$\L$ is a proper subset of $\L_j$, we may deduce from Lemma~\ref{lem-grow} that the growth rate 
of $\L^*$ is strictly smaller than~$\gamma$. Let $\gamma_\L$ be the growth rate of~$\L^\ast$. 

Let $Z$ denote the word $X_jY_j$. From the knowledge of $j$ and $Z$, we can uniquely 
recover the first $j-1$ slices $W_1,\dotsc,W_{j-1}$, and therefore also $Y_j$, $p_j$, $s_j$ 
and~$\L$. Also, $\weight(Z)$ can take at most $2K-1$ values, since
\[
n(1-1/Q)-K<\weight(Z)<n(1-1/Q)+K.
\] 
It follows that there are $\cO(n^{D-1}\gamma^{n(1-1/Q)})$ possible choices of $j$ and $Z$, and for 
each such choice, no more than $\cO(\gamma_\L^{n/Q})$ choices for~$W_j$. In total there are at most 
\[
 \cO(n^{D-1}\gamma^{n(1-1/Q)}\gamma_\L^{n/Q})=o(|\C_n|/n)
\]
possible words $W\in\C_n$ that have a free slice $W_j$ with~$\L_j\neq\coL_j$. 

Focus now on the situation when $\L_j=\coL_j=\L$ for every $j$ such that $W_j$ is a free slice. 
Let us fix a value of $j\in[Q]$, and let us prove that there are at most $O(|\C_n|/n)$ 
words $W\in\C_n(j)$ for which $d_X+d_Y < D+1$. 
We already know from the previous arguments that we may restrict our attention to 
cases when $W_j$ is inside a dominant loop block both in the prefix and the suffix decomposition, 
which implies that $d_X$ and $d_Y$ are both nonzero.

Define $m=n/Q$. Note that $\weight(X_j)$ can only take one of the $K$ values in the range 
$[(j-1)m,(j-1)m+K)$, and similarly, $\weight(Y_j)$ is in the range $((Q-j)m-K,(Q-j)m]$ and 
$\weight(W_j)$ is in $(m-K,m+K)$.

From Lemma~\ref{lem-grow}, the number of possible choices for $X_j$ and $Y_j$ is, respectively, 
$\Theta(m^{d_X-1}\gamma^{(j-1)m})$ and $\Theta(m^{d_Y-1}\gamma^{(Q-j)m})$. Together with the 
$\Theta(\gamma^{m})$ choices for $W_j\in\L^\ast_m$, this yields no more than 
$\cO(m^{d_X+d_Y-2}\gamma^{Qm})\le\cO(|\C_n|/n)$ possibilities for a word $W\in\C_n(j)$ that fails to 
satisfy $d_X+d_Y=D+1$. 

By Lemma~\ref{lem-tight}, the above conditions already imply that $p_j$ and $s_j$ match with 
probability $1-\cO(1/n)$. 

To prove the last condition of typicality, choose a uniformly random $W\in\C_n(j)$ for some fixed 
$j$. Let $W'_j$ be a word over $\L^\ast$ that has the same weight and the same final symbol as 
$W_j$, and define $W'=X_jW'_jY_j$. Note that $W'$ again belongs to $\C_n(j)$ and that $W'_j$ is its 
$j$-th slice (the reason we require that $W'_j$ has the same final symbol as $W_j$ is to ensure that 
the boundary between the $j$-th and $(j+1)$-st slice is preserved). In particular, for a uniformly 
random $W\in\C_n(j)$, the word obtained from $W_j$ by removing its last symbol is a uniformly random 
word over $\L^\ast$, that is, any two words from $\L^\ast$ of the same weight are equally likely to 
be obtained this way. From Proposition~\ref{pro-occur}, we then deduce that, for $n$ large enough 
and up to exponentially small probability, for every word $P\in\L^*$ of constant length, $W_j$ has 
at least $\sqrt n$ disjoint $P$-blocks. Since with probability $1-\cO(1/n)$, $\L$ is 
equal to $\L_j$, the last condition of typicality follows.
\end{proof}

We are now almost ready to prove Theorem~\ref{thm-finind}. For the final argument, we will 
distinguish two cases. First, we will deal with classes of exponential growth, i.e., those 
with~$\gamma>1$. Equivalently, those are the classes whose dominant loop alphabets contain more than 
one element. 

Next, we will handle the classes with $\gamma=1$, i.e., those whose every loop alphabet is either 
empty or contains the single symbol~$1$. In order to have unbounded growth, such a class must then
satisfy $D>1$. 

\newcommand{\cont}{\widehat}

\begin{proposition}
\label{pro-exp} 
Any permutation class $\C$ with growth rate greater than 1 and with finitely many 
sum-indecomposables exhibits a Wilf collapse.
\end{proposition}
\begin{proof}
It is enough to show that every typical word $W\in\C_n$ belongs to a Wilf class of size 
$2^{\Omega(\sqrt n)}$. We will assume throughout that $n$ is large enough in comparison to the 
constants $K$, $D$ and~$Q$. 

\begin{figure}
 \includegraphics[width=\textwidth]{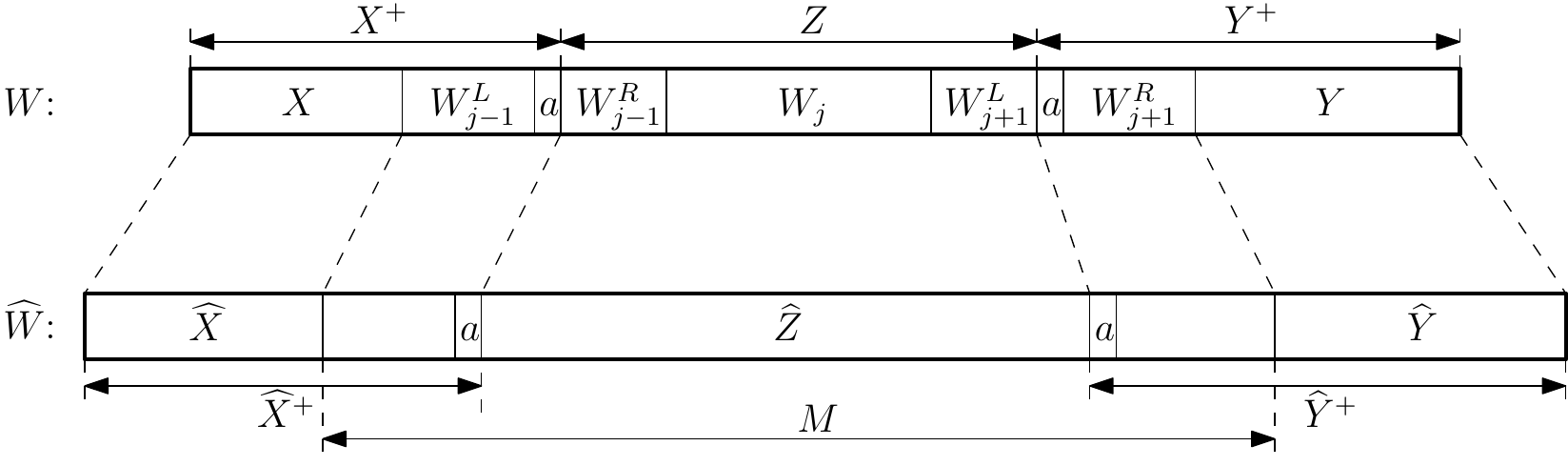}
\caption{Illustration of the proof of Proposition~\ref{pro-exp}}\label{fig-exp}
\end{figure}

Choose a typical word $W\in\C_n$, and let $W_1W_2\dotsb W_Q$ be its equitable partition. By the 
choice of $Q$, we know that there is an index $j\in\{2,\dotsc,Q-1\}$ such that the three slices 
$W_{j-1}$, $W_j$ and $W_{j+1}$ are all free. Write $W$ as $W=XW_{j-1}W_jW_{j+1}Y$, with $X=W_1\dotsc 
W_{j-2}$ and $Y=W_{j+2}\dotsb W_Q$; see Figure~\ref{fig-exp}. Let $p$ be the prefix state of $X$ 
and $s$ the suffix state of~$Y$.  

By typicality, we know that $p$ and $s$ match, and therefore they share a common loop 
alphabet~$\L$. Moreover, $p$ and $s$ are dominant, and therefore $\L^\ast$ has growth rate 
$\gamma>1$. It follows that $|\L|\ge 2$. Let $a\in\L$ be a maximal symbol of~$\L$ in the 
containment relation.

Write $W_{j-1}$ as a concatenation of the form $W_{j-1}^LaW_{j-1}^R$, where $W_{j-1}^R$ is the 
longest suffix of $W_{j-1}$ that has no occurrence of $a$. Note that such decomposition is 
possible, 
since $W_{j-1}$ contains the symbol $a$ by typicality. In fact, $W_{j-1}^L$ contains at least 
$\sqrt{n}-1$ disjoint occurrences of $aaaa$, and therefore has weight more than $K$, for $n$ large 
enough. Symmetrically, we partition $W_{j+1}$ as $W_{j+1}^L a W_{j+1}^R$, with $W_{j+1}^L$ being the 
longest prefix with no occurrence of the symbol~$a$. Define now $X^+=XW_{j-1}^La$, 
$Y^+=aW_{j+1}^RY$, and $Z=W_{j-1}^R W_{j} W_{j+1}^L$, so that the word $W$ can be written as 
$W=X^+ Z Y^+$. 

We claim that if $Z'\in\L^\ast$ is Wilf-equivalent to $Z$ in the class $\L^\ast$, then $W=X^+ Z Y^+$ 
is Wilf-equivalent to $W'=X^+ Z' Y^+$ in~$\C$. To see this, assume that $\Phi$ is a weight-preserving
bijection mapping words in $\L^\ast$ containing $Z$ to those that contain~$Z'$. We now describe a 
weight preserving bijection from words of $\C$ containing $W$ to those that contain~$W'$. 

Let $\cont W\in\C$ be a word that contains~$W$. Let $\cont X$ be the shortest prefix of $\cont W$ 
that contains $X$, and $\cont Y$ the shortest suffix of $\cont W$ containing~$Y$. Let $M$ 
be the `middle' part of $\cont W$ between $\cont X$ and $\cont Y$, i.e., $\cont W=\cont X
M\cont Y$. Observe that all the symbols of $M$ belong to $\L$: if $M$ contained a 
symbol $c\not\in \L$, then $\cont W$ would contain $XcY$ as a subword; however, since the prefix 
state of $X$ matches the suffix state of $Y$, $XcY$ is not in~$\C$. This is a contradiction, as 
$\cont W$ is in~$\C$.

Let $\cont X^+$ be the shortest prefix of $W$ containing $X^+$. Clearly, $\cont X$ is a prefix of 
$\cont X^+$. It is possible that in an embedding of $X^+$ into $\cont X^+$, one or 
more initial symbols of $W_{j-1}^L$ get mapped to the last symbol of $\cont X$. However, since 
$W_{j-1}^L$ has weight greater than $K$, it cannot be fully contained in the last symbol of $\cont 
X$, and in particular, some of its symbols get mapped into~$M$. Consequently, the final symbol of 
$X^+$ (which is the symbol $a$) gets mapped to a symbol of~$M$. Since $a$ is a maximal symbol 
of $\L$, and $M$ only contains symbols from $\L$, we conclude that the final symbol of $\cont X^+$ 
is also the symbol~$a$. Symmetrically, let $\cont Y^+$ be the shortest suffix of $\cont W$ 
containing~$Y^+$. We again conclude that the first symbol of $\cont Y^+$ is the symbol~$a$.

Let $\cont p$ be the prefix state of $\cont X^+$ and $\cont s$ the suffix state of~$\cont Y^+$. We 
claim that both these states have loop alphabet~$\L$. To see this, let $U$ be a word obtained by 
concatenating the maximal symbols of $\L$ in any order. By typicality, $W_{j-1}$ has at least 
$\sqrt n$ disjoint $U$-blocks, of which at least $\sqrt n-1$ are in $W_{j-1}^L$ (recall that 
$W_{j-1}^R$ has no occurrence of the maximal symbol $a$). When embedding 
$X^+$ into $\cont X^+$, at most $K$ of these $U$-blocks can be embedded into the last symbol of 
$\cont X$, but for $n$ large enough, at least $Q$ of these $U$ blocks are embedded into $\cont 
X^+\setminus \cont X$ (i.e., the suffix of $\cont X^+$ that follows after $\cont X$). That means 
that $\cont X^+\setminus \cont X$ contains $Q$ disjoint blocks $U_1, \dotsc,U_Q$, each 
containing~$U$ as 
a subword. Since the symbols of $U$ are maximal in $\L$, each $U_i$ must in fact contain the 
symbols of $U$ as a subsequence. Since $\cont X^+$ has at most $Q-1$ prefix state transitions, there 
is a $U_i$ which does not have any state transition. Therefore $U_i$ is in a loop block of 
a prefix state $p'$ whose loop alphabet $\L_{p'}$ contains all the symbols of $U$, and therefore 
also all the symbols of~$\L$. Since $\L$ is a dominant loop alphabet, this means that $\L_{p'}=\L$. 
Since all the symbols of $\cont X^+$ after $U_i$ belong to $\L$, there are no more prefix state 
transitions after $U_i$, and $p'=\cont p$. This shows that $\cont p$ has loop alphabet $\L$, and a 
symmetric argument applies to $\cont s$ as well. 

Let $\cont Z$ be the part of $\cont W$ between $\cont X^+$ and $\cont Y^+$. Since $\cont W$ 
contains $W$, we conclude that $\cont Z$ contains~$Z$. Note that here we use the fact that the last 
symbol of $X^+$ is equal to the last symbol of $\cont X^+$ and similarly for $Y^+$; this guarantees 
that in any 
embedding of $W$ into $\cont W$, no symbol from $Z$ can be mapped to the last symbol of $\cont X^+$ 
or the first symbol of~$\cont Y^+$, and in particular $Z$ gets mapped entirely into~$\cont Z$. We 
now 
define $\cont Z'=\Phi(\cont Z)$ and  $\cont W'=\cont X^+\cont Z' \cont Y^+$. Since $\cont Z'$ 
contains $Z'$, $\cont W'$ 
contains~$W'$. The mapping $\cont W\mapsto\cont W'$ is easily seen to be the required bijection 
from words containing $W$ to words containing~$W'$ in the class~$\C$. 

To prove the proposition, it now suffices to show that there are many words Wilf-equivalent to $Z$ 
in the class~$\L^\ast$. This, however, can be easily done. Recall that $a$ is a maximal symbol of 
$\L$, and let $b$ be any other symbol of $\L$ (here we use that $|\L|>1$). The word $Z$ contains the 
free slice $W_j$, which, by typicality, contains at least $\sqrt n$ disjoint block occurrences of 
the word~$aaba$. By Corollary~\ref{cor-incompatible-shuffle}, replacing any such occurrence by a 
block occurrence of $abaa$ preserves the Wilf class in $\L^\ast$, since $a$ is maximal in $\L$ and 
hence the pairs $(a,b)$ and $(b,a)$ are incompatible. This yields at least $2^{\sqrt n}$ words in 
the $\L^\ast$-Wilf class of~$Z$, and therefore also in the $\C$-Wilf class of any typical word~$W$. 
\end{proof}

Let us say that a permutation $\C$  with finitely many sum-indecomposables is an \emph{unbounded 
polynomial class} if its growth rate $\gamma$ is equal to 1 and its growth function is unbounded. 
For the rest of this section, we will only consider unbounded polynomial classes. Since any such 
class $\C$ has growth rate $1$, it follows that every dominant prefix or suffix state has loop 
alphabet $\{1\}$, while every non-dominant state has empty loop alphabet. In particular, any 
$W\in\C$ has fewer than $Q$ symbols not belonging to dominant loop blocks, and all these symbols 
are transition symbols. Consequently, each $\C_n$ has only a bounded number of words that have at 
most one dominant loop block in their prefix or suffix decomposition. Since $|\C_n|$ is unbounded, 
it follows that $D>1$. Observe that Proposition~\ref{pro-growth} implies that with probability 
$1-\cO(1/n)$, a uniformly random $W\in\C_n$ has $D$ dominant loop blocks in both its prefix and its 
suffix decomposition.

Let us say that a loop block in the prefix or suffix decomposition of a word $W\in\C_n$ is 
\emph{large} if it has length (or equivalently weight) at least $2KQ+1$. We say that a letter in a 
loop block is \emph{central} if the loop block contains at least $KQ$ letters preceding it and also 
at least $KQ$ letters following it. In particular, each large block has at least one central letter.

\begin{lemma}\label{lem-bigblocks} Let $\C$ be an unbounded polynomial class. With probability 
$1-\cO(1/n)$, in a uniformly random $W\in\C_n$, all the dominant loop blocks in the prefix and 
suffix decomposition are large. 
\end{lemma}
\begin{proof}
Recall the notation $\C_n(T,\tau)$ from Proposition~\ref{pro-growth}. Let 
$\C_n(T,\tau,i,j)$ denote the set of those elements of $\C_n(T,\tau)$ whose 
$i$-th dominant prefix loop block has weight~$j$. It follows from the 
calculations in the proof of Proposition~\ref{pro-growth} that 
$\C_n(T,\tau,i,j)$ has size $O(n^{d_T-2})$ for any fixed $i$ and~$j$, where 
$d_T$ is the number of dominant states in~$T$. Summing these contributions over 
all $T$, $\tau$, $i\le D$ and $j\le 2KQ$, we conclude that there are at most 
$O(n^{D-2})$ words in $\C_n$ that have a small dominant loop block.
\end{proof}

\begin{lemma}\label{lem-dw} Let $\C$ be an unbounded polynomial class. Let $W\in\C_n$ be a word 
with $k$ dominant loop blocks in its prefix decomposition, and suppose that all these loop blocks 
are large. Let $\cont W\in\C$ be a word that contains~$W$. Then $\cont W$ has at least $k$ 
dominant loop blocks in its prefix decomposition. Moreover, if $\cont W$ has exactly $k$ dominant 
loop blocks, then in every embedding of $W$ into $\cont W$, for any dominant loop block $B_i$ of 
$W$, all the central symbols of $B_i$ are mapped to loop symbols of~$\cont W$. Analogous properties 
hold for suffix decompositions as well.
\end{lemma}
\begin{proof} Let $B_1,\dotsc,B_k$ be the dominant loop blocks in the prefix decomposition of~$W$. 
For $i\in[k-1]$, let $t_i$ be the symbol of $W$ immediately following~$B_i$. Necessarily, $t_i$ is 
different from 1. In particular, for every embedding of $W$ into $\cont W$, the symbol $t_i$ gets 
mapped to a transition symbol $\cont t_i$ of~$\cont W$. For $2\le i\le k-1$, let $\cont W_i$ be the 
subword of $\cont W$ between $\cont t_{i-1}$ and $\cont t_i$, including the two symbols $\cont 
t_{i-1}$ and $\cont t_i$ themselves; we also define $\cont W_1$ as the prefix of $\cont W$ ending in 
$\cont t_1$ and $\cont W_k$ as the suffix of $\cont W$ starting in $\cont t_{k-1}$. Since $B_i$ is 
mapped into $\cont W_i$ and $B_i$ is large, $\cont W_i$ has at least $2Q+1$ symbols, and therefore 
each $\cont W_i$ has a loop symbol. Since each $\cont W_i$ begins or ends with a transition symbol, 
$\cont W$ has at least $k$ nonempty loop blocks, which are necessarily dominant.

Suppose now that $\cont W$ has exactly $k$ dominant loop blocks. It follows that each $\cont W_i$ 
contains a unique dominant loop block $\cont B_i$ of~$\cont W$. Therefore, all the symbols of $\cont 
W_i$ except perhaps the leftmost $Q$ and the rightmost $Q$ belong to the loop block~$\cont B_i$. 
Since the block $B_i$ is mapped into $\cont W_i$, all its central symbols must get mapped into 
$\cont B_i$.
\end{proof}

\begin{proposition}
\label{pro-poly}
Any unbounded polynomial class $\C$ exhibits a Wilf collapse.
\end{proposition}
\begin{proof}
Choose a uniformly random word $W\in\C_n$. With probability at least $1-\cO(1/n)$, the word is 
typical, has $D$ dominant loop blocks in both the prefix and the suffix decomposition, and each of 
these loop blocks is large. By typicality, we know that $W$ has a sequence of three slices 
$W_{j-1}W_jW_{j+1}$ that are all contained in a single prefix loop block~$B$. We may assume, 
without loss of generality, that $B$ is not the rightmost dominant prefix loop block of $W$: 
if $B$ were the rightmost dominant loop block, we would consider suffix decompositions 
instead of prefix ones, and apply the following argument symmetrically.

\begin{figure}
 \includegraphics[width=\textwidth]{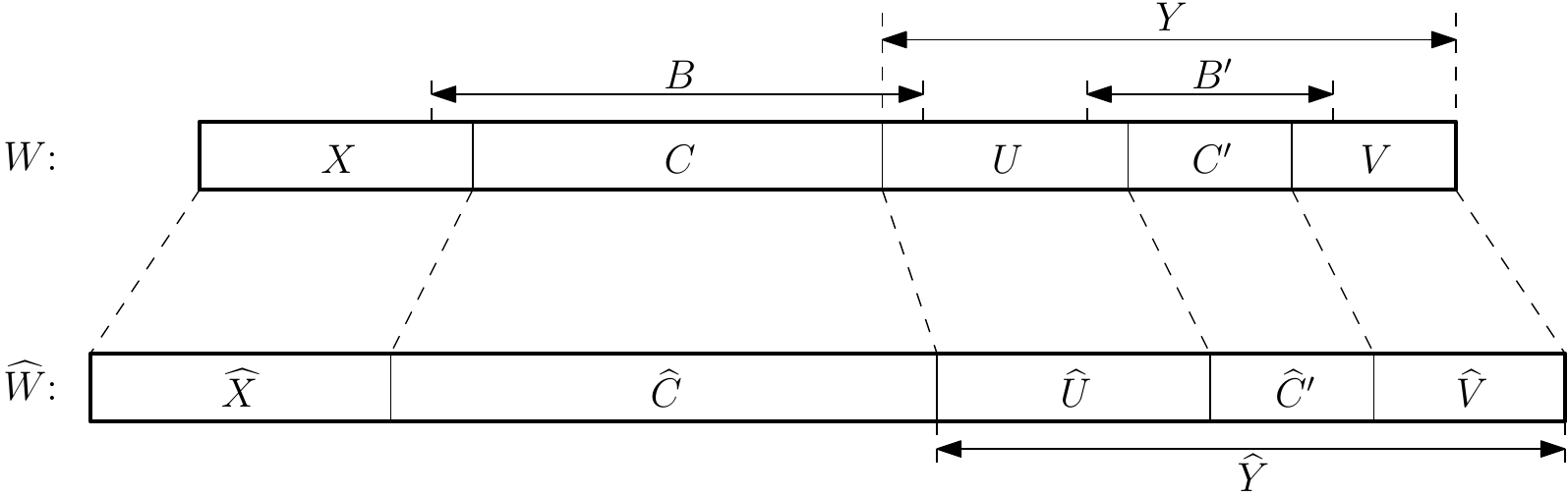}
\caption{Illustration of the proof of Proposition~\ref{pro-poly}}\label{fig-pol}
\end{figure}

Let $X$ be the prefix of $W$ that contains all the symbols preceding $B$ and the first $2KQ+1$ 
symbols of~$B$. See Figure~\ref{fig-pol}. Let $Y$ be the suffix of $W$ containing the rightmost 
$2KQ+1$ symbols of~$B$ 
and all the symbols to the right of~$B$. Let $C$ be the sequence of symbols of $B$ that are 
neither in $X$ nor in~$Y$. Note that for $n$ large enough, the slice $W_j$ is entirely contained in 
$C$, and in particular $C$ is nonempty and we may write~$W=XCY$.

Let $d_X$ be the number of dominant loop blocks in the prefix decomposition of $X$, and let $d_Y$ be 
the number of dominant loop blocks in the suffix decomposition of~$Y$. By typicality, we have 
$d_X+d_Y=D+1$. Since $B$ is not the rightmost dominant prefix loop block, we get $d_X\le D-1$, 
and hence $d_Y\ge 2$. Let $B'$ be the leftmost dominant suffix loop block of $Y$ which is 
disjoint from~$B$. Let $C'$ be the subword of $B'$ consisting of its central elements. We may 
then write $Y$ as $Y=UC'V$, with $U$ and $V$ being the symbols of $Y$ before and after $C'$, 
respectively.

Let $m$ be the length of $C$ and $m'$ the length of $C'$. Note that $m\ge \weight(W_j)=\Theta(n)$ and 
$m'\ge 1$. Let us now fix a value $k\in[m]$, and let $W'$ be the word obtained from $W$ by removing 
$k$ symbols from the block $B$ and inserting these $k$ symbols into~$B'$ (necessarily all 
these $k$ symbols are copies of the symbol `$1$'). We will now show that $W$ is Wilf-equivalent to 
$W'$ in the class~$\C$, implying that $W$ belongs to a Wilf class of size~$\Omega(n)$. 

Let $\cont W\in\C$ be a word containing~$W$. Let $\cont X$ be the shortest prefix of $\cont W$ 
containing $X$, let $\cont Y$ be the shortest suffix of $\cont W$ containing~$Y$, and let $\cont C$ 
be the symbols of $\cont W$ between $\cont X$ and~$\cont Y$. Noting that all the dominant prefix 
loop blocks of $X$ and all the dominant suffix loop blocks of $Y$ are large, we may apply 
Lemma~\ref{lem-dw} to conclude that $\cont X$ has at least $d_X$ dominant prefix loop blocks, and 
$\cont Y$ has at least $d_Y$ dominant suffix loop blocks. In fact, since $\cont X\cont Y$ is in $\C$ 
and $d_X+d_Y=D+1$, we conclude by Lemma~\ref{lem-xydom} that $\cont X$ has exactly $d_X$ dominant 
prefix loop blocks and $\cont Y$ exactly $d_Y$ dominant suffix loop blocks. By the second part of 
Lemma~\ref{lem-dw}, in any embedding of $Y$ into $\cont Y$, all the central symbols of the dominant 
loop blocks, and in particular all the symbols of $C'$, get mapped to loop symbols. Moreover, 
all the symbols of $\cont C$ are loop symbols in the prefix decomposition of $\cont W$, since if 
$\cont C$ contained a transition symbol, then $\cont W$ would have more dominant prefix loop blocks 
than $W$, which is impossible. 

It follows that $\cont W$ can be written as $\cont W=\cont X\cont C\cont Y$, where $\cont C$ is a 
sequence of length at least $m$ in which all symbols are equal to~$1$. Moreover, $\cont Y$ can be 
further written as $\cont Y=\cont U\cont C'\cont V$, where $\cont V$ is the shortest suffix of 
$\cont Y$ that contains $V$, $\cont C'$ is a sequence of length $m'$ whose all symbols are equal to 
$1$, and $\cont U$ contains~$U$. We may now transform $\cont W$ into a word $\cont W'$ by moving 
$k$ symbols from $\cont C$ to~$\cont C'$. Then $\cont W'$ belongs to $\C$, since it only differs 
from $\cont W$ by the length of its loop blocks, $\cont W'$ clearly contains $W'$, and we easily see 
that the map $\cont W\mapsto \cont W'$ is a weight-preserving bijection between the words of $\C$ 
containing $W$ and those containing~$W'$. 

This shows that $W$ and $W'$ are indeed equivalent in~$\C$, and the Wilf class of $W$ has 
size at least $m=\Theta(n)$. It follows that $\C$ exhibits a Wilf collapse.
\end{proof}

Propositions~\ref{pro-exp} and \ref{pro-poly} together establish Theorem~\ref{thm-finind}.


\section{Concluding remarks}
\label{sec-concluding}

We have demonstrated that Wilf collapse occurs in a wide variety of permutation classes. As 
mentioned in the introduction, the only ingredients we seem to need to trigger such a collapse are 
a form of greedy embedding for detecting permutation involvement, together with a representation in 
terms of words that combines with the greedy embedding to allow for local symmetries 
that guarantee Wilf equivalence.

A notable example where our methods of establishing Wilf collapse fail is $\Av(321)$ -- the 
class of permutations containing no occurrence of a 321 pattern. In \cite{Guillemot} (see also 
\cite{AlbertLackner}) a greedy approach to detecting involvement is described in this class but the 
complexity of the ways in which sum-indecomposable permutations can be combined here (along perhaps 
with the failure of super-criticality) have stymied our attempts to prove a Wilf collapse in 
$\Av(321)$. Furthermore, empirical evidence for this class suggests that if a collapse does occur it 
is far less ``robust'' than we see in our other examples -- the largest observed Wilf classes are 
those containing the permutations of the form $(d+1)(d+2) \cdots n \, 12 \cdots d$ (and some others) 
previously considered in \cites{MansourV, Mansour}.  

A related permutation class, the class of ``skew-merged'' permutations (permutations that can be 
written as the merge of a decreasing and an increasing subsequence) has none of the nice closure 
properties that we might hope for, but again the existence of a greedy algorithm for pattern 
detection and an underlying structure of ``spirals'' might yield a Wilf collapse. 

\begin{question}
Does $\Av(321)$ have a Wilf collapse? Does the class of skew-merged permutations?
\end{question}

%


\bibliographystyle{plain}
\bibliography{collapse}

\begin{bibdiv}
\begin{biblist}

\bib{AlbertLi}{article}{
      author={{Albert}, M.},
      author={{Li}, J.},
       title={{Wilf-collapse in permutation classes having two basis elements
  of size three}},
        date={2017-10},
     journal={ArXiv:1710.04107},
      eprint={1710.04107},
}

\bib{AlbertBouvel}{article}{
      author={Albert, Michael},
      author={Bouvel, Mathilde},
       title={A general theory of {W}ilf-equivalence for {C}atalan structures},
        date={2015},
        ISSN={1077-8926},
     journal={Electron. J. Combin.},
      volume={22},
      number={4},
       pages={Paper 4.45, 29},
      review={\MR{3441665}},
}

\bib{AlbertLackner}{article}{
      author={Albert, Michael},
      author={Lackner, Marie-Louise},
      author={Lackner, Martin},
      author={Vatter, Vincent},
       title={The complexity of pattern matching for 321-avoiding and
  skew-merged permutations},
        date={2016},
        ISSN={1365-8050},
     journal={Discrete Math. Theor. Comput. Sci.},
      volume={18},
      number={2},
       pages={Paper No. 11, 17},
      review={\MR{3597961}},
}

\bib{Dembo}{book}{
      author={Dembo, Amir},
      author={Zeitouni, Ofer},
       title={Large deviations techniques and applications},
      series={Stochastic Modelling and Applied Probability},
   publisher={Springer-Verlag, Berlin},
        date={2010},
      volume={38},
        ISBN={978-3-642-03310-0},
         url={https://doi.org/10.1007/978-3-642-03311-7},
        note={Corrected reprint of the second (1998) edition},
      review={\MR{2571413}},
}

\bib{DFLS}{article}{
      author={Duchon, Philippe},
      author={Flajolet, Philippe},
      author={Louchard, Guy},
      author={Schaeffer, Gilles},
       title={Boltzmann samplers for the random generation of combinatorial
  structures},
        date={2004},
        ISSN={0963-5483},
     journal={Combin. Probab. Comput.},
      volume={13},
      number={4-5},
       pages={577\ndash 625},
         url={https://doi.org/10.1017/S0963548304006315},
      review={\MR{2095975}},
}

\bib{Flajolet}{book}{
      author={Flajolet, Philippe},
      author={Sedgewick, Robert},
       title={Analytic combinatorics},
   publisher={Cambridge University Press, Cambridge},
        date={2009},
        ISBN={978-0-521-89806-5},
         url={https://doi.org/10.1017/CBO9780511801655},
      review={\MR{2483235}},
}

\bib{Guillemot}{incollection}{
      author={Guillemot, Sylvain},
      author={Vialette, St\'{e}phane},
       title={Pattern matching for 321-avoiding permutations},
        date={2009},
   booktitle={Algorithms and computation},
      series={Lecture Notes in Comput. Sci.},
      volume={5878},
   publisher={Springer, Berlin},
       pages={1064\ndash 1073},
         url={https://doi.org/10.1007/978-3-642-10631-6_107},
      review={\MR{2792803}},
}

\bib{Higman}{article}{
      author={Higman, Graham},
       title={Ordering by divisibility in abstract algebras},
        date={1952},
        ISSN={0024-6115},
     journal={Proc. London Math. Soc. (3)},
      volume={2},
       pages={326\ndash 336},
         url={https://doi.org/10.1112/plms/s3-2.1.326},
      review={\MR{0049867}},
}

\bib{Hoeffding}{article}{
      author={Hoeffding, Wassily},
       title={Probability inequalities for sums of bounded random variables},
        date={1963},
        ISSN={0162-1459},
     journal={J. Amer. Statist. Assoc.},
      volume={58},
       pages={13\ndash 30},
  url={http://links.jstor.org/sici?sici=0162-1459(196303)58:301<13:PIFSOB>2.0.CO;2-D&origin=MSN},
      review={\MR{0144363}},
}

\bib{Mansour}{incollection}{
      author={Mansour, Toufik},
       title={321-avoiding permutations and {C}hebyshev polynomials},
        date={2004},
   booktitle={Mathematics and computer science. {III}},
      series={Trends Math.},
   publisher={Birkh\"{a}user, Basel},
       pages={37\ndash 38},
      review={\MR{2090493}},
}

\bib{MansourV}{article}{
      author={Mansour, Toufik},
      author={Vainshtein, Alek},
       title={Layered restrictions and {C}hebyshev polynomials},
        date={2001},
        ISSN={0218-0006},
     journal={Ann. Comb.},
      volume={5},
      number={3-4},
       pages={451\ndash 458},
         url={https://doi.org/10.1007/s00026-001-8021-9},
        note={Dedicated to the memory of Gian-Carlo Rota (Tianjin, 1999)},
      review={\MR{1897635}},
}

\bib{Vatter-survey}{incollection}{
      author={Vatter, Vincent},
       title={Permutation classes},
        date={2015},
   booktitle={Handbook of enumerative combinatorics},
      series={Discrete Math. Appl. (Boca Raton)},
   publisher={CRC Press, Boca Raton, FL},
       pages={753\ndash 833},
      review={\MR{3409353}},
}

\end{biblist}
\end{bibdiv}

\end{document}